\documentclass[12pt]{amsart}
\usepackage{graphicx} % Required for inserting images
\usepackage{amssymb,latexsym,eucal}

\usepackage{hyperref}
\usepackage[letterpaper,centering,text={6.5in,9in}]{geometry}
\usepackage{latexsym}
\usepackage[utf8]{inputenc}
\usepackage{amsmath}
\usepackage{amsthm}
\usepackage{bigints}
\usepackage{amsfonts}
\usepackage{amssymb}
\usepackage{epsfig}
\usepackage{nicefrac}
\usepackage[all]{xy}
\usepackage{graphicx}
\usepackage{enumerate}
\usepackage{enumitem}
\usepackage{mathtools}    
\DeclarePairedDelimiterX\setc[2]{\{}{\}}{\,#1 \;\delimsize\vert\; #2\,}
\usepackage{float}
\usepackage{pgfplots}
\usepackage{bbm}
\usepackage{caption,subcaption}
\usepackage{breqn}
\usepackage{amsaddr}
\usepackage{tikz}
\usetikzlibrary{
intersections, arrows.meta,
automata,er,calc,
backgrounds,
mindmap,folding,
patterns,
decorations.markings,
fit,decorations.pathmorphing,
shapes,matrix,
positioning,
shapes.geometric,
arrows,through, 
}
\usepackage{tikz-cd}

\def\bigmid{\ \rule[-3.5mm]{0.1mm}{9mm}\ }

\newtheorem{theorem}{Theorem}[section]

\newtheorem{proposition}[theorem]{Proposition}
\newtheorem{lemma}[theorem]{Lemma}

\theoremstyle{definition}

\newtheorem{remark}[theorem]{Remark}

\newtheorem{definition}[theorem]{Definition}

\newcommand{\RR}{{\mathbb R}}

\newcommand{\QQ}{{\mathbb Q}}

\newcommand{\ton}{{\otimes_{\Lambda_{\geq 0}}}}

\newcommand{\lr}{{\,\,\longrightarrow\,\,}}

\title[A lower bound for relative symplectic cohomology barcode entropy]{A lower bound for relative symplectic cohomology barcode entropy}
\author[Jonghyeon Ahn]{Jonghyeon Ahn}
\newcommand{\Addresses}{{
\bigskip
\bigskip
\footnotesize
\textsc{Institute for Basic Science, Center for Geometry and Physics, Pohang, 37673, South Korea}\par\nopagebreak
\textit{E-mail address}: \texttt{jahn@ibs.re.kr}}}

\date{}

\begin{document}

\maketitle
\begin{abstract}
   In this paper, we continue to study the barcode entropy of relative symplectic cohomology $SH_M(K)$ of a Liouville domain $K$ embedded in a symplectic manifold $M$. This barcode entropy measures the exponential growth rate of the number of not-too-short bars in the persistence module $SH_M(K)$. We prove that this Floer-theoretic invariant admits a nontrivial lower bound in terms of the topological entropy of the Reeb flow on $\partial K$ when the Reeb flow possesses a hyperbolic invariant set. More precisely, we show that the barcode entropy of $SH_M(K)$ is bounded below by the topological entropy of the Reeb flow restricted to a hyperbolic invariant set.
\end{abstract}

\maketitle
\tableofcontents

\section{Introduction}
\subsection{Motivation}

The \textit{barcode entropy} of a persistence module, which was introduced by Cineli, Ginzburg, and Gürel \cite{cgg}, measures the exponential growth rate of the number of ``not-too-short'' bars in the barcode of the persistence module. In many cases, barcode entropy has been shown to be closely related to the underlying dynamics.

In their original work \cite{cgg}, Cineli, Ginzburg, and Gürel studied the persistence module associated to the Lagrangian Floer cohomology of pairs of Hamiltonian isotopic Lagrangian submanifolds. From this construction, one obtains the barcode entropy $\hbar(\varphi_H)$ of a Hamiltonian diffeomorphism $\varphi_H$. A central result of their work is that $\hbar(\varphi_H)$ is strongly related to the dynamics of $\varphi_H$. More precisely, they proved that
\begin{align}\label{int1}
    \hbar(\varphi_H)\leq h_{\mathrm{top}}(\varphi_H),
\end{align}
where $h_{\mathrm{top}}(\varphi_H)$ denotes the topological entropy of $\varphi_H$. Moreover, they established a nontrivial lower bound for $\hbar(\varphi_H)$ in terms of the dynamics of $\varphi_H$, namely,
\begin{align}\label{int2}
    \hbar(\varphi_H)
    \geq
    h_{\mathrm{top}}(\varphi_H|_V)
\end{align}
whenever $V$ is a hyperbolic set for $\varphi_H$. Analogous results to \eqref{int1} and \eqref{int2} have since been established in many other Floer-theoretic settings; see \cite{a2,cggm,f,f2,fls,ggm,m}.

Among many Floer-theoretic settings, we focus on the \textit{relative symplectic cohomology}, introduced by Varolgunes \cite{v}. For a closed symplectic manifold $(M,\omega)$ and a compact subset $K \subset M$, Varolgunes constructed the relative symplectic cohomology $SH_M(K)$, which carries the relative symplectic information of $K$ in $M$. The barcode entropy $\hbar(SH_M(K))$ of $SH_M(K)$ was introduced in \cite{a2}, and in the case where $K\subset M$ is a Liouville domain, it was shown to admit an upper bound in terms of the topological entropy of the Reeb flow on the boundary $\partial K$, analogous to \eqref{int1}. The goal of this paper is to derive a nontrivial lower bound for $\hbar(SH_M(K))$, providing a counterpart of \eqref{int2} in the setting of relative symplectic cohomology.

\subsection{Main result} In this subsection, we state the main result of the paper, which provides a nontrivial lower bound for the barcode entropy of relative symplectic cohomology in terms of the topological entropy of the Reeb flow restricted to a hyperbolic invariant set. Thus, the dynamical complexity of the Reeb flow is reflected in the persistence structure of relative symplectic cohomology.

For a closed symplectic manifold $(M,\omega)$ and a compact subset $K\subset M$, Varolgunes \cite{v} introduced the relative symplectic cohomology $SH_M(K)$ of $K$ in $M$. Our setting for viewing this relative symplectic cohomology as a persistence module is as follows. We assume that $\omega|_{\pi_2(M)}=0$ so that the action functional of a Hamiltonian orbit is well-defined. Moreover, the subset $K$ is a Liouville domain, meaning that there exists a vector field $X$ defined on $K$ such that $\mathcal{L}_X\omega = \omega$.  The boundary \(\partial K\) naturally carries a contact structure whose contact form is given by $ \alpha = \iota_X \omega|_{\partial K}$. Finally, we assume that the boundary $\partial K$ is incompressible, namely,  the homomorphism $ \pi_1(\partial K)\lr\pi_1(M)$ induced by inclusion is injective. The incompressibility of $\partial K$ implies that a loop in $K$ is contractible in $K$ if and only if it is contractible in $M$.

Following \cite{cgg}, we define the \textbf{relative symplectic cohomology barcode entropy} of $SH_M(K)$, denoted by $\hbar(SH_M(K))$. Roughly speaking, this invariant measures the exponential growth rate of the number of sufficiently ``not-too-short" bars in the barcode of the persistence module $SH_M(K)$. For any $\epsilon>0$, let $$b_\epsilon(\textrm{tru}(SH_M(K)),\sigma)$$ be the number of bars of the truncation of the persistence module $SH_M(K)$ at $\sigma$ whose lengths are greater than $\epsilon$. Note that $b_\epsilon(\textrm{tru}(SH_M(K)),\sigma)$ is an increasing function of $\sigma$ and a decreasing function of $\epsilon$. We define
\begin{align*}
    \hbar_\epsilon(SH_M(K)) = \limsup_{\sigma \to \infty} \frac{1}{\sigma}\log^+b_\epsilon(\textrm{tru}(SH_M(K)),\sigma),
\end{align*}
and
\begin{align*}
    \hbar(SH_M(K)) = \lim_{\epsilon \to 0} \hbar_\epsilon(SH_M(K))
\end{align*}
where $\log^+a = \max\{\log_2 a, 0\}$.

\begin{theorem}[Theorem \ref{mainthm}]\label{thmb}
    Let $(M,\omega)$ be a closed symplectic manifold and  $K \subset M$ be a Liouville domain with incompressible boundary. Assume that $\omega|_{\pi_2(M)}=0$. Let $\varphi_\alpha^t$ be the Reeb flow on $(\partial K, \alpha)$. Then
    \begin{align*}
        h_{\textrm{top}}(\varphi^t_\alpha|_V) \leq \hbar(SH_M(K)) 
    \end{align*}
    whenever $V$ is a compact hyperbolic invariant set for $\varphi_\alpha^t$.
\end{theorem}

The main idea of the proof of Theorem \ref{thmb} is as follows. For a suitable Hamiltonian function $H$, we divide its Hamiltonian orbits into two classes: lower orbits and upper orbits. We prove that there exists no Floer trajectory of $H$ which is asymptotic to an upper orbit at $-\infty$ and to a lower orbit at $+\infty$ (Proposition \ref{nof}). Using this fact, we show that only lower orbits contribute to the filtered relative symplectic cohomology (Lemma \ref{lem1}, Lemma \ref{lem2} and Lemma \ref{lem3}). Moreover, we prove that the filtered relative symplectic cohomology can be represented, with control of the action window, by the Floer cohomology generated by the lower orbits of a single Hamiltonian function (Theorem \ref{a2t}). The proof of the main result (Theorem \ref{mainthm}) then closely follows the argument in \cite{cggm2}.
\\

\noindent
\textbf{Organization of the paper.} Since the material presented in this paper draws on several areas of mathematics, we begin in Section \ref{sec:section2} by reviewing the basic notions and facts from each subject that will be used throughout the paper. 
In Section \ref{sec:section3}, we reformulate the definition of relative symplectic cohomology for a Liouville domain by establishing restrictions on the behavior of Floer trajectories for suitably chosen Hamiltonian functions. We then show that the barcode entropy can be described in terms of the barcode entropy associated to a single Hamiltonian function. Finally, the proof to establish the lower bound for the relative symplectic cohomology barcode entropy, Theorem \ref{thmb}, is carried out in Section \ref{sec:section4}.

\section{Preliminaries}\label{sec:section2}
\subsection{From topological data analysis} We briefly review persistence modules, originally developed in topological data analysis; most of the exposition is adapted from \cite{csgo, prsz}. Let us fix a ground field $\mathbbm{k}$ and we usually drop this from our notation. For our purposes, we restrict attention to persistence modules satisfying the following conditions.

\begin{definition}
A \textbf{persistence module} is a pair $(V, \pi)$ where $V$ is a collection $\{V_\tau\}_{\tau\ \in\RR}$ of finite dimensional vector spaces over $\mathbbm{k}$ and $\pi = \{\pi_{\tau_1\tau_2} : V_{\tau_1} \lr V_{\tau_2}\}_{\tau_1, \tau_2 \in \RR}$ is a collection of linear maps from $V_{\tau_1}$ to $V_{\tau_2}$ for $\tau_1\leq \tau_2$ satisfying the following:
    \begin{enumerate}[label=(\alph*)]
        \item (Persistence) For any $\tau \in \RR$, $\pi_{\tau\tau} = \textrm{id}_{V_\tau}$ and for any $\tau_1 \leq \tau_2 \leq \tau_3$, $\pi_{\tau_1 \tau_3} = \pi_{\tau_2 \tau_3} \circ \pi_{\tau_1 \tau_2}$.
        \begin{align*}
            \includegraphics[]{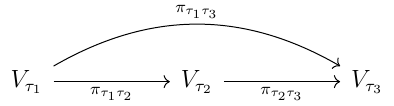}
        \end{align*}
        \item (Locally constant) There exists a closed, bounded from below and nowhere dense subset $\textrm{Spec}(V, \pi) \subset \RR$, which is called the \textbf{spectrum} of $(V, \pi)$ such that $\pi_{\tau_1 \tau_2} : V_{\tau_1} \lr V_{\tau_2}$ is an isomorphism whenever $\tau_1$ and $\tau_2$ are in the same connected component of $\RR \setminus \textrm{Spec}(V,\pi)$.
        \item (Semicontinuity) For any $\tau \in \RR$ and any $\tau' \leq \tau$ sufficiently close to $\tau$, the map $\pi_{\tau' \tau} : V_{\tau'} \lr V_\tau$ is an isomorphism.
        \item (Lower bound) There exists $\tau_0 \in \RR$ such that $V_\tau = 0$ for $\tau\leq \tau_0$.
 \end{enumerate}
    We will usually use the notation $V$ to denote the persistence module $(V,\pi)$ when it is clear from the context.
\end{definition}

One important example of a persistence module is an interval persistence module. For an interval $I = (a,b], $ where $-\infty < a < b \leq\infty$, the \textbf{interval persistence module} $\mathbbm{k}I = (V, \pi)$ is defined by
\begin{align*}
    V_\tau = \begin{cases}
        \mathbbm{k} & \text{if}\,\,\tau \in I \\
        0 & \text{if} \,\, \tau\notin I,
    \end{cases}
    \,\,\,\,\textrm{and}\,\,\,\,\pi_{\tau_1 \tau_2} = \begin{cases}
        \textrm{id}_{\mathbbm{k}} & \text{if}\,\,\tau_1,\tau_2 \in I\\
        0 &\text{otherwise.}
    \end{cases}
\end{align*}

A fundamental result in the theory of persistence modules is the \textit{normal form theorem}, asserting that every persistence module admits a decomposition into interval modules. The precise statement is as follows.
\begin{theorem}[Normal form theorem]\label{nft}
    Let $(V, \pi)$ be a persistence module. Then there exists a countable collection $\{ (I_i, m_i)\}_{i=1,2,3,\cdots}$ of intervals $I_i$ with multiplicity $m_i$ such that
    \begin{align*}
        (V, \pi) \cong \bigoplus_{i=1}^\infty \left(\mathbbm{k} I_i\right)^{m_i}.
    \end{align*}
    Moreover, this decomposition is unique up to permutation, that is, to any persistence module $(V,\pi)$, there exists a unique multiset $\mathcal{B}(V,\pi)$ consisting of intervals $I_i$ with multiplicity $m_i$. This multiset $\mathcal{B}(V,\pi)$ is called the \textbf{barcode} of the persistence module $(V,\pi)$ and an element of $\mathcal{B}(V,\pi)$ is called a \textbf{bar} of $(V,\pi)$.
\end{theorem}
We close this subsection by defining the barcode entropy of a persistence module. Before giving the definition, we introduce a useful operation on persistence modules. Let $\sigma\in \RR$. The \textbf{truncation} $\text{tru}(V,\pi ;\sigma ) = (\text{tru}(V;\sigma), \text{tru}(\pi;\sigma))$ of $(V,\pi)$ at $\sigma$ is defined by
        \begin{align*}
            \text{tru}(V;\sigma)_\tau = \begin{cases}
                V_\tau &\text{if}\,\, \tau<\sigma\\
                0 &\text{otherwise}
            \end{cases}
            \,\,\,\,\textrm{and}\,\,\,\,\text{tru}(\pi;\sigma)_{\tau_1 \tau_2} = \begin{cases}
                \pi_{\tau_1 \tau_2}&\text{if}\,\, \tau_2< \sigma\\
                0 &\text{otherwise}.
            \end{cases}
         \end{align*}
       
\begin{definition}\label{bcd}
    Let $(V, \pi)$ be a persistence module. For any $\epsilon>0$, let $ \mathcal{B}_{\epsilon}(V)$ be the set of barcodes of $V$ whose lengths are greater than $\epsilon$ and let $b_{\epsilon}(V)$ be the number of bars in $\mathcal{B}_{\epsilon}(V)$, i.e., $b_{\epsilon}(V) = | \mathcal{B}_{\epsilon}(V)|$. The \textbf{$\epsilon$-barcode entropy} $\hbar_{\epsilon}(V)$ of $V$ is defined to be
\begin{align*}
    \hbar_{\epsilon}(V) = \limsup_{\sigma \to \infty} \frac{1}{\sigma}\log^+ b_{\epsilon}(\textrm{tru}(V,\sigma)),
\end{align*}
where $\log^+a = \max\{\log_2 a, 0\}$. The \textbf{barcode entropy} $\hbar(V)$ of $V$ is defined by
    \begin{align*}
        \hbar(V) = \lim_{\epsilon \to 0} \hbar_{\epsilon}(V).
    \end{align*}
\end{definition}

\subsection{From Floer theory} In this subsection, we review the basic aspects of Floer theory needed later in the paper. For more detailed exposition, see \cite{sala}.
\subsubsection{Hamiltonian Floer cohomology}
The \textbf{Novikov field} $\Lambda$ is defined by
\begin{align*}
    \Lambda = \left\{ \sum_{i=1}^\infty a_i T^{\lambda_i} \bigmid a_i \in \QQ, \lambda_i \in \RR\,\,\text{and}\,\, \lim_{i \to \infty} \lambda_i = \infty \right\} 
\end{align*}
where $T$ is a formal variable. There is a \textbf{valuation map} $\textrm{val} : \Lambda \to \RR \cup \{\infty\}$ given by
\begin{align*}
    \textrm{val} (x) = 
    \begin{cases}
      \displaystyle\min_{i} \{\lambda_i \mid a_i \neq 0 \} \,&\text{if}\,\, x = \displaystyle\sum_{i=1}^\infty a_i T^{\lambda_i} \neq 0\\
      \infty \,&\text{if}\,\, x = 0.
    \end{cases}   
\end{align*}
For any $r \in \RR$, define $\Lambda_{\geq r} = \textrm{val}^{-1}([r,\infty])$. In particular, we call
\begin{align*}
    \Lambda_{\geq 0} = \left\{ \sum_{i=1}^\infty a_i T^{\lambda_i} \in \Lambda \bigmid  \lambda_i \geq 0 \right\}
\end{align*}
the \textbf{Novikov ring}. 

Let $(M, \omega)$ be a symplectic manifold and let $H : S^1 \times M \to \RR$ be a Hamiltonian function on $M$. The \textbf{Hamiltonian vector field} $X_H$ of $H$ is defined by 
\begin{align*}
    \iota_{X_H} \omega = dH.
\end{align*}
We say that a Hamiltonian function $H :S^1 \times M \to \RR$ is \textbf{nondegenerate} if every 1-periodic orbit $x$ of $X_H$ is nondegenerate, that is, for the Hamiltonian flow $\phi_H^t$ of $H$, the map $$(d \phi^1_H)_{x(0)} : T_{x(0)}M \lr T_{x(1)}M = T_{x(0)}M$$ has no eigenvalue equal to 1. We denote the set of all nondegenerate contractible 1-periodic orbits of $X_H$ by $\mathcal{P}(H)$. For $x \in \mathcal{P}(H)$, let $c_x : D^2 \lr M$ be a disk capping of $x$. The \textbf{action functional} of the capped orbit $(x, c_x)$ is defined by
\begin{align}\label{act}
    \mathcal{A}_H(x, c_x) = \int_{D^2} c_x^* \omega + \int_{S^1} H(t, x(t))dt.
\end{align}

The \textbf{Floer complex} of $H$ is defined by
\begin{align*}
    CF^*(H) = \bigoplus_{x \in \mathcal{P}(H)} \Lambda_{\ge 0} \langle x \rangle.
  \end{align*}
Fix a generic almost complex structure $J$ on $M$ compatible with $\omega$, namely,
\begin{align*}
    g : TM \times TM \lr \RR, \,\,g(v,w) = \omega(v,Jw)
\end{align*}
is a Riemannian metric on $M$. For $x, y \in \mathcal{P}(H)$, let $\pi_2(M; x, y)$ be the set of homotopy classes of smooth maps from $\RR \times S^1$ to $M$ asymptotic to $x$ and $y$ at $-\infty$ and $+\infty$, respectively. For $x,y \in \mathcal{P}(H)$ and $A \in \pi_2(M;x,y)$, the moduli space $\mathcal{M}(H,J; x, y; A)$ consists of smooth maps $u : \RR \times S^1 \lr M$ satisfying the Floer equation
\begin{align}\label{fee}
    \partial_s u + J(u) \left( \partial_t u - X_{H}(u) \right) = 0,
\end{align}
where $\partial_s = \frac{\partial }{\partial s}$ and $\partial_t = \frac{\partial }{\partial t}$, together with the asymptotic conditions
\begin{align*}
    \lim_{s \to -\infty} u(s,t) = x(t) \,\,\textrm{and} \,\,\lim_{s \to \infty} u(s,t) = y(t).
\end{align*}
The \textbf{Floer differential} $d : CF^*(H) \lr CF^{*}(H)$ of the Floer complex $CF(H)$ is given by
\begin{align*}
    d x = \sum_{\substack{y \in \mathcal{P}(H)  \\A \in \pi_2(M; x, y)}} \#\left(\mathcal{M}(H,J ; x, y; A)/\RR\right)\, T^{E_{\text{top}}(u)} y
\end{align*}
where $\#\left(\mathcal{M}(H,J ; x, y; A)/\RR\right)$ denotes the count of Floer trajectories modulo the $\RR$-action $(s_0 \cdot u)(s,t) = u(s+s_0,t) $ and the \textbf{topological energy} $E_{\text{top}}(u)$ of a Floer trajectory $u$ is given by
\begin{align*}
    E_{\text{top}}(u) = \int_{S^1} H(t,y(t)) dt - \int_{S^1} H(t,x(t)) dt + \omega(A).
\end{align*}
Since the Floer differential $d$ satisfies $d \circ d = 0$, the \textbf{Floer cohomology} $HF^*(H)$ of $H$ is defined to be 
\begin{align*}
    HF^*(H) = H^* \left( CF(H), d \right).
\end{align*}
Moreover, the Floer cohomology of $H$ with coefficients in $\Lambda$ is defined by
\begin{align*}
    HF^*(H ; \Lambda) = H^*(CF(H)\ton \Lambda) = HF^*(H)\ton \Lambda. 
\end{align*}

For a Floer trajectory $u: \RR \times S^1 \lr M$ connecting $x$ to $y$ and capped orbits $(x,c_x)$ and $(y,c_y)$ with $c_y = c_x \# u$ where $\#$ denotes the connected sum, a standard computation shows that
\begin{align}\label{acf}
    E_{\text{top}}(u)& = \mathcal{A}_H(y,c_y) - \mathcal{A}_H(x,c_x)\nonumber\\
    &= \int_{\RR \times S^1}\omega \left(\partial_su, J(u)\partial_su\right) ds\wedge dt
    \\&\ge0\nonumber
\end{align}
where the last inequality follows from the compatibility of $J$ with $\omega$.

For two Hamiltonian functions $H_0$ and $H_1$ with $H_0 \leq H_1$, there exists a \textbf{continuation map} 
\begin{align*}
   c^{H_0, H_1} : CF^*(H_0) \lr CF^*(H_1)
\end{align*}
defined by an appropriate count of Floer trajectories of a monotone homotopy $H_s$ from $H_0$ to $H_1$. As in the definition of the Floer differential, each Floer trajectory is weighted by $T^{E_{\mathrm{top}}(u)}$, where
\begin{align*}
    E_{\text{top}}(u) &= \int_{S^1}H_1(t,y(t)) dt - \int_{S^1} H_0(t,x(t)) dt + \omega(A).
   \end{align*}
   Equivalently, if $c_x$ and $c_y$ are cappings satisfying $c_y=c_x\#u$, then
\begin{align}\label{acc}
    E_{\text{top}}(u) &= \mathcal{A}_{H_1}(y,c_y) - \mathcal{A}_{H_0}(x,c_x)\nonumber\\
    &=\int_{\RR \times S^1}\omega \left(\partial_su, J(u)\partial_su\right) ds\wedge dt  +\int_{\RR \times S^1}(\partial_sH_s)(t,u)ds\wedge dt\\
    &\ge 0,\nonumber
\end{align}
where the last inequality follows from the compatibility of $J$ with $\omega$ and the monotonicity of the homotopy $H_s$. Moreover, the continuation map induces a map on cohomology
\begin{align*}
    c^{H_0,H_1}:HF(H_0)\lr HF(H_1).
\end{align*}
\subsubsection{Relative symplectic cohomology}
We recall the definition of relative symplectic cohomology introduced by Varolgunes in \cite{v}. Further details and applications can be found in \cite{a,a22,dgpz,msv,sun}.

A key ingredient in the construction of relative symplectic cohomology is the completion of the Floer complex. For a $\Lambda_{\geq0}$-module $A$, its completion is defined as follows. For $r' >r$, there exists a map 
\begin{align*}
    \Lambda_{\geq 0} / \Lambda_{\geq r'} \lr \Lambda_{\geq 0} / \Lambda_{\geq r} 
\end{align*}
and this map induces a map
\begin{align}\label{inverse}
    A \ton \Lambda_{\geq 0} / \Lambda_{\geq r'} \lr A \ton \Lambda_{\geq 0} / \Lambda_{\geq r}.
\end{align}
Then $\left\{ A \ton \Lambda_{\geq 0} / \Lambda_{\geq r} \right\}_{r>0}$ forms an inverse system with the maps given by \eqref{inverse}. The \textbf{completion} $\widehat{A}$ of $A$ is defined by
\begin{align*}
    \widehat{A} = \varprojlim_{r \to\infty} A \ton \Lambda_{\geq 0} / \Lambda_{\geq r}.
\end{align*}

Notice that the map \eqref{inverse} is surjective because the tensor product is a right-exact functor and the inverse system $\left\{ A \ton \Lambda_{\geq 0} / \Lambda_{\geq r} \right\}_{r>0}$ satisfies the Mittag--Leffler condition. Therefore, if the sequence
\begin{align*}
    0\lr A \ton  \Lambda_{\geq 0} / \Lambda_{\geq r} \lr B \ton  \Lambda_{\geq 0} / \Lambda_{\geq r} \lr C \ton  \Lambda_{\geq 0} / \Lambda_{\geq r}\lr 0
\end{align*}
is exact for each $r>0$, then 
\begin{align*}
    0 \lr \widehat{A} \lr \widehat{B} \lr  \widehat{C} \lr0
\end{align*}
is also exact by the Mittag--Leffler theorem for the inverse limit.

Let $(M, \omega)$ be a closed symplectic manifold and let $K \subset M$ be a compact subset. Choose a sequence of nondegenerate Hamiltonian functions $\{H_\ell : S^1 \times M \to \RR\}_{\ell=1,2,3,\cdots}$ such that $H_\ell|_{S^1\times K}$ is negative for all $\ell=1,2,3,\cdots$, $H_1\leq H_2 \leq H_3 \leq \cdots$, and 
\begin{align*}
    \lim_{\ell \to \infty} H_\ell(t,p) = \begin{cases}
        0&\textrm{if} \,\,p \in K\\ \infty&\textrm{if}\,\,p \notin K.
    \end{cases}
\end{align*}
This sequence induces a direct system
\begin{align*}
    CF(H_1) \lr CF(H_2) \lr CF(H_3) \lr\cdots
\end{align*}
where all the horizontal arrows are continuation maps. The \textbf{relative symplectic cohomology} $SH_M(K)$ of $K$ in $M$ is defined by
\begin{align}\label{reldef}
    SH^*_M(K) = H^* \left( \widehat{\varinjlim_{\ell\to\infty}} CF(H_\ell) \right)
\end{align}
where $\displaystyle\widehat{\varinjlim_{\ell\to\infty}} CF(H_\ell)$ denotes the completion of $\displaystyle\varinjlim_{\ell\to\infty} CF(H_\ell)$. Similarly, the relative symplectic cohomology with coefficients in $\Lambda$ is defined by
\begin{align*}
    SH^*_M(K;\Lambda) = H^* \left( \widehat{\varinjlim_{\ell\to\infty}} CF(H_\ell) \ton \Lambda\right) = H^* \left( \widehat{\varinjlim_{\ell\to\infty}} CF(H_\ell) \right) \ton \Lambda.
\end{align*}

\subsection{From dynamical systems} In this subsection, we recapitulate some facts from the theory of dynamical systems that we need throughout this paper. For a comprehensive exposition, see \cite{kh, th}.
\subsubsection{Topological entropy} Let $(X,d)$ be a compact metric space and let $f : X \lr X$ be a continuous map. Define an increasing sequence $\{d_\ell^f\}$ of metrics for $\ell=1,2,3,\cdots$ by
\begin{align}\label{metric}
    d_\ell^f(x,y) = \max_{0 \leq i\leq \ell} d(f^i(x),f^i(y))
\end{align}
where $f^i$ denotes the composition of $f$ with itself $i$ times. Let $ B_f(x,\epsilon,\ell)$ be the open ball of radius $\epsilon$ using the metric $d_\ell^f(x,y)$. A subset $E \subset X$ is said to be \textbf{$(\ell,\epsilon)$-spanning} if 
\begin{align*}
    X \subset \bigcup_{x \in E} B_f(x,\epsilon,\ell).
\end{align*}
Let $S_d(f,\epsilon,\ell)$ be the minimal cardinality of an $(\ell,\epsilon)$-spanning set. Then the \textbf{topological entropy} $h_{\textrm{top}}(f)$ of $f$ is defined by
\begin{align*}
   h_{\textrm{top}}(f) = \lim_{\epsilon \to 0} \limsup_{\ell \to \infty} \frac{1}{\ell} \log S_d(f,\epsilon,\ell).
\end{align*}

The definition of the topological entropy for a flow is completely analogous. Let $F=\{f^t\}_{t \in \RR }$ be a continuous flow on $X$. Define the family of metrics $\{ d_T^{F}\}_{T \in \RR}$ by
\begin{align*}
    d_T^{F} (x,y) = \max_{0\leq t \leq T} d(f^t(x),f^t(y)).
\end{align*}
The topological entropy $h_{\textrm{top}}(F) = h_{\textrm{top}}(f^t)$ of the flow $F$ is defined by
\begin{align*}
    h_{\textrm{top}}(F) = \lim_{\epsilon \to 0} \limsup_{T \to \infty} \frac{1}{T} \log S_d(F,\epsilon,T).
    \end{align*}
    where the quantity $S_d(F,\epsilon,T)$ is defined analogously as above. Note that the topological entropies of a map and a flow do not depend on the choice of the metric $d$. Moreover, the topological entropy of a flow agrees with that of its time-one map: $h_{\textrm{top}}(F) = h_{\textrm{top}}(f^1)$. 
\subsubsection{Hyperbolicity} In the study of dynamical systems, hyperbolic sets occupy a central position due to their rich structure and numerous remarkable properties. We begin by recalling the definition of hyperbolicity.
\begin{definition}
    Let $M$ be a smooth manifold and $F =\{f^t\}_{t \in \RR} $ be a smooth flow on $M$. An $F$-invariant set $V \subset M$ is said to be a \textbf{hyperbolic} set for $F$ if there exist a Riemannian metric on an open neighborhood $U$ of $V$ and constants $0<\lambda<1<\mu$ such that for each $x \in V$, the tangent space $T_xM$ admits a splitting $$T_xM = E_x^0 \oplus E_x^s \oplus E_x^u$$ with the following properties:
    
    \begin{itemize}
        \item $E_x^0$ is the one-dimensional subspace spanned by $\displaystyle\left.\frac{d}{dt}\right|_{t=0} f^t(x)$.\\
        \item $d f^t_x\left( E_x^{s} \right) = E_{f^t(x)}^{s}$ and $d f^t_x\left( E_x^{u} \right) = E_{f^t(x)}^{u}$ for all $t \in \RR$.\\
        \item $\| d f^t_x|_{E_x^s} \| \leq \lambda^t$ for $t>0$ and $\| d f_x^{-t}|_{E_x^u} \| \leq \mu^{-t}$ for $t>0$.\vspace{0.2 cm}
    \end{itemize}
\end{definition}
Hyperbolic sets possess many remarkable properties. Among them, we recall below one property that is relevant for our purposes. In particular, hyperbolic sets provide an alternative characterization of topological entropy. For a smooth flow $F=\{f^t\}_{t \in \RR}$, we say that a compact $F$-invariant set $V$ is \textbf{locally maximal} if there exists a neighborhood $U$ of $V$, called an \textbf{isolating neighborhood}, such that 
\begin{align*}
    V = \bigcap_{t \in \RR} f^t(U).
\end{align*}
In other words, the set $V$ is the maximal $F$-invariant subset of $U$, or equivalently, for every $x \in U\setminus V$, the image $\{f^t(x)\mid t \in \RR \}$ is not entirely contained in $U$.

We are now ready to present an alternative method for computing topological entropy.
\begin{theorem}\label{alttop}
    Let $F=\{f^t\}_{t \in \RR}$ be a smooth flow on a Riemannian manifold $M$ and $V \subset M$ be a locally maximal hyperbolic set of $F$. Then
    \begin{align*}
        h_{\textrm{top}}(f^t|_V) = \limsup_{\sigma \to \infty} \frac{1}{\sigma} \log p_F(\sigma)
    \end{align*}
    where $p_F(\sigma)$ is the number of periodic orbits of $F$ of period up to $\sigma$.
\end{theorem}
\section{Relative symplectic cohomology as a persistence module}\label{sec:section3}
In this section, we focus on the relative symplectic cohomology in the case where the subset $K$ is a \textbf{Liouville domain} embedded in a closed symplectic manifold $(M,\omega)$: A subset $K$ of $M$ is called a compact domain with contact type boundary if it is a codimension zero submanifold of $M$ with boundary and there exists an outward pointing vector field $X$ defined on a neighborhood of $\partial K$ satisfying $\mathcal{L}_X \omega = \omega$; such a vector field $X$ is called a \textbf{Liouville vector field}. If the Liouville vector field $X$ is defined on $K$, then we call $K$ a \textbf{Liouville domain}. The boundary \(\partial K\) naturally carries a contact structure whose contact form $\alpha$ is given by
\begin{align*}
    \alpha = \iota_X \omega|_{\partial K}.
\end{align*}
We denote by $R_\alpha$ the associated Reeb vector field and by $\varphi_\alpha^t$ the corresponding Reeb flow on the contact manifold $(\partial K, \alpha)$. Using the flow of the Liouville vector field, a neighborhood of $\partial K$ in $M$ may be identified with a symplectic collar
$$\partial K \times [1,1+3r_1]$$
for some $r_1>0$. For future purposes, we choose $r_1>0$ sufficiently small so that the larger region $\partial K \times [1-r_1,1+4r_1]$ is still embedded in $M$. On this collar, the symplectic form is given by $d (r \alpha)$ where $\alpha $ is the contact form on $\partial K$ and $r$ is the coordinate on $[1,1+3r_1]$.

Throughout the paper, we pick a generic almost complex structure $J$ which is compatible with $\omega$. Furthermore, in the case where $K \subset M$ is a compact domain with contact type boundary, we may choose  $J$ to be \textbf{cylindrical}. Namely,
\begin{align*}
        J \displaystyle( r\partial_r )= R_\alpha \,\,\text{and $J$ preserves}\,\, \xi
    \end{align*}
where $\partial_r = \frac{\partial}{\partial r}$ and $\xi=\text{Ker}(\alpha)$.
\subsection{Semi-admissible Hamiltonian functions} In the case where the subset $K$ is compact domain with contact type boundary, we can make use of a useful class of Hamiltonian functions to define $SH_M(K)$. We say that a Hamiltonian function $H : S^1 \times M \to \RR$ is \textbf{$K$-semi-admissible} if it satisfies the following.\vspace{0.2 cm}
\begin{enumerate}[label=(A\arabic*)]
    \item $H=0$ on $S^1 \times K$.\vspace{0.2 cm}
    \item $H(t, p, r) =  h(r)$ on $S^1 \times \left(\partial K \times [1, 1+3r_1]\right)$ such that\vspace{0.3 cm}
    \begin{itemize}
        \item $ h(r)=h_1(r)$ on $S^1 \times \left(\partial K \times [1, 1+r_1]\right)$ for some strictly convex and increasing function $h_1$,\vspace{0.2 cm}
        \item $h(r) = s_Hr - i_H$ on $S^1 \times (\partial K \times [1+r_1, 1+2r_1])$ where the slope $s_H$ of $H$ satisfies that $s_H \notin \text{Spec}(\partial K, \alpha)$ and $i_H \in \RR$. Here, $\text{Spec}(\partial K, \alpha)$ denotes the set of all periods of contractible Reeb orbits of $(\partial K, \alpha)$, and\vspace{0.2 cm}
        \item $h(r) = h_2(r)$ on $S^1 \times \left(\partial K \times [1+2r_1,1+3r_1]\right)$ for some strictly concave and increasing function $h_2$.\vspace{0.2 cm}
    \end{itemize}
        \item $H= c_H$ for some constant $c_H >0$ on $S^1 \times \left(M - (K \cup \left(\partial K \times [1,1+3r_1]\right)\right))$.\vspace{0.2 cm}
\end{enumerate}

For a $K$-semi-admissible Hamiltonian function $H$ given by $h(r)$ on $S^1 \times( {\partial K} \times[1, 1+3r_1])$, the Hamiltonian vector field $X_H$ is given by
\begin{align*}
  X_H(p, r) = -h'(r) R_\alpha(p)  
\end{align*}
and hence, every $T$-periodic Reeb orbit $\gamma$ of $(\partial K, \alpha)$ corresponds to a 1-periodic Hamiltonian orbit $x = (\gamma, r_*)$ of $H$ with $h'(r_*) = T$ and they travel in the opposite direction. This implies that
\begin{align}\label{peri}
    \int_{S^1}  \gamma^*\alpha = - \mathcal{A}(\gamma)
\end{align}
where $\mathcal{A}(\gamma)$ denotes the period of $\gamma$.

For a $K$-semi-admissible Hamiltonian function $H$, the Floer cohomology $HF(H)$ of $H$ is defined as the Floer cohomology $HF(\widetilde{H})$ of a $C^2$-small nondegenerate perturbation $\widetilde{H}$ of $H$ leaving the linear part unchanged. 

It was shown in \cite{a2} that, in this setting, the relative symplectic cohomology can be described using \(K\)-semi-admissible Hamiltonians. More precisely, if \(H\) is a \(K\)-semi-admissible Hamiltonian function, then
\begin{align}\label{ivo}
    SH^*_M(K)\cong H^*\left(\widehat{\varinjlim_{\sigma\to\infty}}CF(\sigma H)\right).
\end{align}

Since a $K$-semi-admissible Hamiltonian function $H$ is constant outside $\partial K \times [1,1+3r_1]$, we have a positive lower bound for the topological energy of Floer trajectories of $H$ in certain special situations. More precisely, we have the following proposition.

\begin{proposition}\label{mon}
    Let $H_0$ and $H_1$ be $K$-semi-admissible Hamiltonian functions with $H_0 \le H_1$, and let $H_s$ be a monotone homotopy from $H_0$ to $H_1$ through $K$-semi-admissible Hamiltonian functions. Then there exists $\eta_1 >0$ such that 
    \begin{align*}
        E_{\textrm{top}}(u) > \eta_1
    \end{align*}
    for every Floer trajectory $u : \RR \times S^1 \lr M$ of $H_s$ with $\textrm{Im}(u) \not\subset \partial K \times[1-r_1,1+4r_1]$.
\end{proposition}
\begin{proof}
Recall the monotonicity theorem for holomorphic curves: there exist constants $c>0$ and $r_0>0$, depending only on $(M,\omega,J)$, such that for any holomorphic curve $u : (\Sigma,j) \lr (M,\omega,J)$ passing through the point $p$,
    \begin{align*}
        \int_{\Sigma_r} u^*\omega \geq cr^2
    \end{align*}
    where $\Sigma_r = u^{-1}(B_{r}(p))$ for $0<r\le r_0$. Here, $B_r(p)$ is a closed ball in $M$ of radius $r$ centered at $p$. See, for example, \cite{ms}. We may assume that $r_0 < \frac{1}{2}r_1$ Since the image of $u$ is not entirely contained in $\partial K \times[1-r_1,1+4r_1]$, we can choose a point $p \in \textrm{Im}(u)$ such that 
    \begin{align*}
        p \in \partial K \times \left\{1-\frac{1}{2}r_1,1+\frac{7}{2}r_1 \right\}.
    \end{align*}
    Since $H_s$ is constant outside $\partial K \times [1,1+3r_1] $, the restriction $u: \Sigma_{r_0} \lr M$ is a $J$-holomorphic curve where $\Sigma_{r_0} = u^{-1}(B_{r_0}(p))$. Then
    \begin{align*}
        E_{\textrm{top}}(u)& \geq \int_{\Sigma_{r_0}} \omega\left(\partial_s u, J(u)\partial_s u\right) ds\wedge dt&&\textrm{by \eqref{acc}}\\& = \int_{\Sigma_{r_0}} \omega\left(\partial_s u, \partial_t u\right) ds\wedge dt &&\textrm{since $u$ is $J$-holomorphic}\\& = \int_{\Sigma_{r_0}} u^* \omega\\
        &\ge cr_0^2.
    \end{align*}
    Taking $\eta_1 = cr_0^2$, the proof is complete.\\
\end{proof}

On the other hand, when the image of a Floer trajectory is entirely contained in $\partial K \times [1-r_1,1+4r_1]$, we have the following observation.
\begin{proposition}\label{pe}
Let $H_0$ and $H_1$ be $K$-semi-admissible Hamiltonian functions with $H_0 \le H_1$, and let $H_s$ be a monotone homotopy from $H_0$ to $H_1$ through $K$-semi-admissible Hamiltonian functions. Let $u:\RR \times S^1 \lr M$ be a Floer trajectory of $H_s$ connecting $x=(\gamma_x,r_x)$ to $y=(\gamma_y,r_y)$. If $\textrm{Im}(u) \subset \partial K \times [1-r_1,1+4r_1]$, then
     \begin{align*}
         \mathcal{A}(\gamma_x) \geq \mathcal{A}(\gamma_y).
     \end{align*}
     Moreover, the equality holds if and only if the corresponding Reeb orbits $\gamma_x$ and $\gamma_y$ coincide.
\end{proposition}
\begin{proof}
    We decompose $u(s,t)=(v(s,t),r(s,t)) \in \partial K \times [1-r_1,1+4r_1]$ for all $(s,t)\in\RR \times S^1$. The Floer equation \eqref{fee} can be written as 
\begin{align*}
    \pi_{\xi}\left(\partial_sv \right) + \alpha\left( \partial_s v \right) R_\alpha + (\partial_s r) \partial_r +J \left\{ \pi_{\xi}\left(\partial_t v \right) + \alpha\left( \partial_t v \right) R_\alpha + (\partial_t r) \partial_r + h_s'(r)R_\alpha\right\} = 0
\end{align*}
where $\pi_\xi : T(\partial K) \lr \xi$ is the projection. With the cylindrical almost complex structure $J$, the equation is written as follows:
    
    \begin{align}\label{reeb1}
        \begin{cases}
            \displaystyle\pi_\xi\left(\partial_s v \right) + J(v)\pi_\xi\left(\partial_t v\right) =0\vspace{0.2cm}\\
            \displaystyle \partial_s r -r \alpha\left(\partial_t v \right) - r h_s'(r) =0\vspace{0.2cm}\\
            \displaystyle \alpha\left(\partial_s v\right) + \frac{\partial_t r}{r} =0.
        \end{cases}
        \end{align}
        Then 
        \begin{align*}
            -\mathcal{A}(\gamma_y) + \mathcal{A}(\gamma_x) &= \int_{\RR \times S^1} v^* d\alpha &&\textrm{by Stokes' theorem and \eqref{peri}}\\ &= \int_{\RR \times S^1} d\alpha\left( \partial_s v,\partial_t v \right) ds \wedge dt\\
            &= \int_{\RR \times S^1} d\alpha\left( \pi_\xi(\partial_s v),\pi_\xi(\partial_t v) \right) ds \wedge dt\\
            &= \int_{\RR \times S^1} d\alpha\left( \pi_\xi(\partial_s v),J(v)\pi_\xi(\partial_s v) \right) ds \wedge dt &&\textrm{by the first equation of \eqref{reeb1}}\\
            &\geq 0
        \end{align*}
        where the last inequality follows from the compatibility of $J$. 
        
        If $\mathcal{A}(\gamma_x) = \mathcal{A}(\gamma_y)$, then the calculation above implies that $\pi_\xi(\partial_sv) = 0$. Hence, we have 
        \begin{align*}
            \partial_sv = \alpha(\partial_sv) R_\alpha,
        \end{align*}
        and then the image of $v$ is contained in the image of a single Reeb orbit. Consequently, $\gamma_x = \gamma_y$. This completes the proof.\\ 
\end{proof}

 \begin{remark}
     Proposition \ref{pe} holds as long as the Floer trajectory $u$ admits the decomposition $u=(v,r)$ for all $(s,t) \in \RR \times S^1$. 
 \end{remark}

\subsection{Action filtrations} For the rest of the paper, we assume that $(M,\omega)$ is a closed symplectic manifold with $\omega|_{\pi_2(M)}=0$, and that $K \subset M$ is a Liouville domain with \textbf{incompressible} boundary, meaning that the map $\pi_1(\partial K) \lr \pi_1(M)$ induced by inclusion is injective. Then the action functional associated to a Hamiltonian function $H$, defined in \eqref{act}, is independent of the choice of disk capping of a $1$-periodic orbit of $H$. Furthermore, the incompressibility of $\partial K$ ensures that a loop in $K$ is contractible in $K$ if and only if it is contractible in $M$. Hence, the action $\mathcal{A}_H(x)$ of a 1-periodic orbit $x=(\gamma,r_*)$ of a $K$-semi-admissible Hamiltonian $H$ is computed as
\begin{align*}
    \mathcal{A}_H(x) = A_h(r_*)
\end{align*}
where $A_h : [1, 1+r_1] \cup[1+2r_1, 1+3r_1]  \to \RR$ is defined by
\begin{align*}
    A_h(r) =
    \begin{cases}
        A_{h_1}(r) = -r h_1'(r) + h_1(r)&\textrm{if}\,\,r \in [1,1+r_1]\\
        A_{h_2}(r) = -r h_2'(r) + h_2(r)&\textrm{if}\,\,r \in [1+2r_1,1+3r_1].
    \end{cases}
    \end{align*}
Observe that 
\begin{align*}
    \frac{d}{dr}\left(-r h'(r) +h(r)\right) = -h'(r) - rh''(r) +h'(r) = - rh''(r).  
\end{align*}
This implies that the function $ A_h(r)$ is decreasing on $[1,1+r_1]$ by the convexity of $h_1$ and is increasing on $[1+2r_1,1+3r_1]$ by the concavity of $h_2$. Therefore, 
\begin{align}\label{ima}
    -i_H \le A_{h_1}(r) \le 0\,\,\textrm{and}\,\,-i_H \le A_{h_2}(r) \le c_H.
\end{align}

For any $K$-semi-admissible Hamiltonian function $H$ and $\tau \in \RR$, let $CF^{>-\tau}(H)$ denote the submodule of $CF(H)$ generated by the $1$-periodic orbits of $H$ whose action is greater than $-\tau$.
This is in fact a subcomplex of $CF(H)$ because, as observed in \eqref{acf}, the action increases along Floer trajectories. Hence, we can define the filtered Floer cohomology $$HF^{>-\tau}(H) = H(CF^{>-\tau}(H),d)$$ of $H$ for any $\tau \in \RR$.

Moreover, for two $K$-semi-admissible Hamiltonian functions $H_0$ and $H_1$ with $H_0 \leq H_1$, the continuation map $c^{H_0,H_1}$ respects the action filtration by \eqref{acc}. Therefore, for every $\tau\in\RR$, there is an induced chain map $$ c^{H_0,H_1} : CF^{>-\tau}(H_0) \lr CF^{>-\tau}(H_1),$$ which in turn induces a map on cohomology $$  c^{H_0,H_1} : HF^{>-\tau}(H_0) \lr HF^{>-\tau}(H_1).$$

In view of \eqref{ivo}, the action filtration of $SH_M(K)$ is given by 
\begin{align*}
    SH^{>-\tau}_M(K) = H \left( \widehat{\varinjlim_{\sigma \to \infty}}CF^{>-\tau}(\sigma H)  \right)
\end{align*}
for any $K$-semi-admissible Hamiltonian function $H$ and for $\tau \in \RR$. See \cite{a2}. The goal of this subsection is to simplify the definition of the filtered relative symplectic cohomology. More precisely, we will represent the action filtration $SH^{>-\tau}_M(K)$ by the Floer cohomology of a single Hamiltonian function within a suitable action window. To proceed further, we introduce some terminology. For a $K$-semi-admissible Hamiltonian function $H$, the 1-periodic Hamiltonian orbits of $H$ fall into the following four types:\vspace{0.2cm}
\begin{itemize}
\item constant Hamiltonian orbits of $H $ contained in $K$,\vspace{0.2cm}
\item nonconstant Hamiltonian orbits arising from the convex part of $H$,\vspace{0.2cm}
\item nonconstant Hamiltonian orbits arising from the concave part of $H$, and\vspace{0.2cm}
\item constant Hamiltonian orbits of $H$ outside $K$.\vspace{0.2cm}
\end{itemize}
We refer to the first two types as \textit{lower orbits} and the last two types as \textit{upper orbits}. Let $x$ be a lower orbit of $H$. If $x$ is a constant orbit, then its action $\mathcal{A}_H(x)=0$. If $x = (\gamma,r_*)$ is a nonconstant orbit, then
\begin{align*}
    \mathcal{A}_H(x) &= -r_* h'(r_*) +h(r_*)\\
    &=  -r_* h'(r_*) +\int_{1}^{r_*} h'(r)dr &&h(1)=0\\
    &\leq -r_* h'(r_*) +\int_{1}^{r_*} h'(r_*)dr&&\textrm{$h$ is convex on $(1,1+r_1)$}\\
    &= - h'(r_*)\\
    &= - \mathcal{A}(\gamma).
\end{align*}
Hence, every lower orbit has nonpositive action. Furthermore, if a lower orbit $x = (\gamma,r_*)$ satisfies $\mathcal{A}_H(x) >-\tau$, then the previous estimate implies that
\begin{align*}
    \mathcal{A}(\gamma) = h'(r_*) < \tau.
\end{align*}
%Now, let  $x$ be an upper orbit of $H$. If $x$ is a constant orbit, then $\mathcal{A}_H(x) = c_H$. If $x = (\gamma,r_*)$ is a nonconstant orbit, then
%\begin{align*}
%      \mathcal{A}_H(x) &= -r_* h'(r_*) +h(r_*)\\
%    &=  -r_* h'(r_*) +c_H- \int_{r_*}^{1+3r_1} h'(r)dr &&h(1+3r_1)=c_H\\
%    &\ge -r_* h'(r_*) +c_H- \int_{r_*}^{1+3r_1} h'(r_*)dr &&\textrm{$h$ is concave on $(1+2r_1,1+3r_1)$}\\
%    &= -r_* h'(r_*) +c_H - (1+3r_1 -r_*)h'(r_*)\\
%    &= c_H - (1+3r_1)h'(r_*)\\
%    &= c_H - (1+3r_1)\mathcal{A}(\gamma)
%\end{align*}
\begin{proposition}\label{nof}
    Let $H_0$ and $H_1$ be $K$-semi-admissible Hamiltonian functions with $H_0 \le H_1$, and let $H_s$ be a monotone homotopy from $H_0$ to $H_1$ through $K$-semi-admissible Hamiltonian functions. If 
\begin{align}\label{princ}
 \partial_sh'_s (r)\geq 0 \,\,\textrm{for}\,\,r\in(1,1+3r_1),   
\end{align}
 then there exists no Floer trajectory of $H_s$ asymptotic to an upper orbit of $H_0$ at $-\infty$ and asymptotic to a lower orbit of $H_1$ at $+\infty$.
\end{proposition}
\begin{proof}
    Let $x$ be an upper orbit of $H_0$ and $y$ be a lower orbit of $H_1$. If $x$ is a constant orbit, then 
    \begin{align*}
        \mathcal{A}_{H_0}(x)= c_{H_0} >0.
    \end{align*}
    Since the action is increasing along a Floer trajectory and $\mathcal{A}_{H_1}(y)\leq 0$ as observed above, there is no Floer trajectory of $H_s$ connecting $x$ to $y$. In the case where $x=(\gamma,r_*)$ is a nonconstant orbit, suppose that there exists a Floer trajectory $u$ of $H_s$ connecting $x $ to $y$, namely,
    \begin{align}\label{fe}
        \partial_su + J(u) \left(\partial_t u -X_{H_s}(u)\right) = 0,
    \end{align}
   with asymptotic conditions
    \begin{align*}
        \lim_{s \to -\infty} u(s,t) = x(t) \,\,\textrm{and} \,\,\lim_{s \to \infty} u(s,t) = y(t).
    \end{align*}
     Since the Floer trajectory $u$ starts from $x=(\gamma,r_*) \in \partial K \times [1+2r_1,1+3r_1]$ and ends at $y \in K \cup\left(\partial K \times [1,1+r_1]\right)$, it must leave the region $\partial K \times [1+2r_1,1+3r_1]$ at some point. More specifically, let
    \begin{align*}
        s_0 =\sup \left\{s\in\RR \bigmid \textrm{Im}(u|_{(-\infty,s) \times S^1}) \subset \partial K \times [1+2r_1,1+3r_1]  \right\}.
    \end{align*}
     For $(s,t) \in (-\infty,s_0)\times S^1$, we can decompose $u(s,t) = (v(s,t),r(s,t)) \in \partial K \times [1+2r_1,1+3r_1]$. The radial coordinate $r(s,t)$ satisfies 
        \begin{align}\label{gt}
            r(s,t) \leq r_*
        \end{align}
        since the image of $u$ cannot leave $K_{r_*} =  K \cup (\partial K \times[1,r_*]) $ by Lemma 3.5 of \cite{gt}. Define a function $f:(-\infty,s_0) \to \RR$ by
        \begin{align*}
            f(s) = \int_{S^1} \ln r(s,t) dt. 
        \end{align*}
       % Then by the restriction \eqref{gt} on $r$, we have 
       % \begin{align}\label{lnr}
       %     f(s) \leq \ln r_*.
       % \end{align}
         By the second equation of \eqref{reeb1}, we compute 
         \begin{align}\label{f'}
            f'(s) = \int_{S^1} \frac{\partial_sr(s,t) }{r(s,t)} dt = \int_{S^1} \left(\alpha\left(\partial_t v(s,t) \right) + h_s'(r(s,t)) \right)dt.
                    \end{align}
        To find an estimate for the first integral of \eqref{f'}, let
        \begin{align*}
            g(s) =\int_{S^1} \alpha\left(\partial_t v(s,t)\right) dt. 
        \end{align*}
        Since $v^*(d \alpha) = d (v^* \alpha)$, we have
        \begin{align*}
             d \alpha(\partial_s v, \partial_t v) =  \frac{\partial}{\partial s} \alpha(\partial_t v) - \frac{\partial}{\partial t} \alpha(\partial_s v).
        \end{align*}
        Using this, we compute
         \begin{align}\label{in}
            g'(s) = \int_{S^1} \frac{\partial }{\partial s} \,\alpha\left(\partial_t v \right) dt =\int_{S^1}\left\{ d \alpha\left(\partial_s v, \partial_t v  \right) + \frac{\partial }{\partial t} \alpha\left(\partial_s v \right)  \right\}dt.
        \end{align} 
        Note that the first integrand of \eqref{in} is nonnegative because
        \begin{align*}
            d \alpha\left(\partial_s v, \partial_t v  \right) &= d \alpha\left(\pi_\xi\left(\partial_s v\right), \pi_\xi\left(\partial_t v\right) \right)\\&=
            d \alpha\left(\pi_\xi\left(\partial_s v\right), J(v)\pi_\xi\left(\partial_s v\right) \right)&&\textrm{by the first equation of \eqref{reeb1}}\\&
            \geq 0
        \end{align*}
        where the last inequality follows from the compatibility of $d \alpha$ with $J$ on $\xi$. The second integral
        \begin{align*}
            \int_{S^1} \frac{\partial}{\partial t} \alpha\left( \partial_s v\right) dt 
        \end{align*}
        of \eqref{in} is zero since $v$ is periodic in $t$. Consequently, $g'(s) \geq 0$ on $(-\infty,s_0)$ and thus 
        \begin{align}\label{pos2}
            g(s) \geq \lim_{s \to -\infty}g(s) = \int_{S^1} \alpha \left(\frac{d\gamma}{dt}\right) dt = -h_0'(r_*).
        \end{align}
       The last equality follows from the fact that the Reeb orbit $\gamma$ is traversed in the opposite direction to $x$ and the period of $\gamma$ is $h_0'(r_*)$.
        To estimate the second term in \eqref{f'}, we use the assumption \eqref{princ}. Since $1+2r_1 \leq r(s,t) \leq r_*\leq 1+3r_1$ for $s< s_0$ and $h$ is concave on $(1+2r_1,1+3r_1)$, it follows that
\begin{align}\label{pos}
    h_s'(r(s,t)) \geq h_0'(r(s,t)) \geq h'_0(r_*).
\end{align}
Combining \eqref{pos2} and \eqref{pos} together, we obtain 
\begin{align*}
    f'(s) \geq -h'_0(r_*) +h'_0(r_*) =0,
\end{align*}
    and therefore $f$ is increasing. Consequently,
    \begin{align*}
        f(s) \geq \lim_{s \to -\infty} f(s) = \ln  r_*.
    \end{align*}
    Moreover, for $s<s_0$, it follows from \eqref{gt} that
    \begin{align*}
        r(s,t) =  r_*.
    \end{align*}
    This implies that the Floer trajectory cannot escape the region $\partial K \times [1+2r_1,1+3r_1]$, which contradicts the assumption that it connects an orbit in $\partial K \times [1+2r_1,1+3r_1]$ to an orbit outside this region. This completes the proof.\\ 
    \end{proof}
 For any $K$-semi-admissible function $H$, let $CF_{U}(H)$ denote the submodule of $CF(H)$ generated by upper orbits of $H$. Proposition \ref{nof} implies that the differential preserves $CF_U(H)$, and hence $CF_U(H)$ is a subcomplex of $CF(H)$. Let $$CF_L(H)= CF(H)/CF_U(H).$$ Note that $CF_L(H)$ is a cochain complex generated by lower orbits of $H$, with differential induced by the Floer differential on $CF(H)$. We denote the cohomologies of $CF_U(H)$ and $CF_L(H)$ by 
 \begin{align*}
     HF_U(H)\,\,\textrm{ and}\,\, HF_L(H),
 \end{align*}
 respectively. The action filtrations 
 \begin{align*}
    HF^{>-\tau}_U(H)\,\,\textrm{ and}\,\, HF^{>-\tau}_L(H)
 \end{align*}
 of $HF_U(H)$ and $HF_L(H)$ can be defined in a natural manner. Furthermore, Proposition \ref{nof} implies that for $H_0 \leq H_1$ and the homotopy $H_s$ satisfying the condition of Proposition \ref{nof}, we have continuation maps 
\begin{align}\label{conti}
    c_{U}^{H_0,H_1} : CF_{U}(H_0) \lr CF_{U}(H_1)\,\,\text{and}\,\, c_{L}^{H_0,H_1} : CF_{L}(H_0) \lr CF_{L}(H_1). 
\end{align}
These continuation maps induce the corresponding maps on cohomology and are compatible with the action filtrations.

The next theorem provides a description of the filtered relative symplectic cohomology in terms of the Floer cohomology generated by the lower orbits of a single $K$-semi-admissible Hamiltonian function, together with precise control of the action filtration.
  \begin{theorem}\label{a2t}
Let $H$ be a $K$-semi-admissible Hamiltonian function. Then we have an isomorphism
     \begin{align*}
        SH^{>-\tau}_M(K) \cong HF_{L}^{>a_{ H}(\tau)}( H)
    \end{align*}
    for $0 \leq \tau \leq s_{H} $ where $a_{H}(\tau) = \left( A_{ h_1} \circ ( h_1')^{-1} \right)(\tau).$
\end{theorem}

The proof of Theorem \ref{a2t} requires the following lemmas.

\begin{lemma}\label{lem1}
For any $K$-semi-admissible Hamiltonian function $H$ and for any $\tau\in\RR\cup \{\infty\}$, we have
\begin{align*}
    \widehat{\varinjlim_{\sigma \to \infty}}CF^{>-\tau}_{U}(\sigma H) = 0.
\end{align*}
\end{lemma}
\begin{proof}
Let $\{\sigma_\ell\}$ be an increasing sequence of positive numbers such that $\displaystyle\lim_{\ell \to \infty} \sigma_\ell = \infty$ and $\sigma_\ell s_H \notin \textrm{Spec}(\partial K, \alpha)$. Let $$c^{\ell}_U = c_U^{\sigma_\ell H, \sigma_{\ell+1}H} : CF^{>-\tau}_U(\sigma_\ell H) \lr CF^{>-\tau}_U(\sigma_{\ell+1}H) $$ be the continuation map explained in \eqref{conti}. To prove the claim, it suffices to show that there exists a constant $c>0$ such that for every $\ell =1,2,3,\cdots$,
\begin{align*}
    c^\ell_U\left(CF^{>-\tau}_U(\sigma_\ell H)\right) \subset T^c CF^{>-\tau}_U(\sigma_{\ell+1}H) 
\end{align*}
by Lemma 5.3 of \cite{dgpz}.

Let $H_s$ be the monotone homotopy connecting $\sigma_\ell H$ to $\sigma_{\ell +1} H$ satisfying the condition of Proposition \ref{nof} and let $u: \RR \times S^1 \lr M$ be a Floer trajectory of $H_s$ which is asymptotic to an upper orbit $x$ of $\sigma_\ell H$ at $-\infty$ and to an upper orbit $y$ of $\sigma_{\ell+1} H$ at $\infty$. We shall prove that the topological energy of $u$ admits a positive lower bound independent of $u$ and $\ell$. 

If the image of $u$ is not entirely contained in $\partial K \times [1-r_1,1+4r_1]$, then, as in Proposition \ref{mon}, the monotonicity theorem for holomorphic curves implies the existence of a constant $\eta_1 >0$ such that $E_{\textrm{top}}(u) > \eta_1$. 

In the case where the image of $u$ is completely contained in $\partial K \times [1-r_1,1+4r_1]$, we decompose $u(s,t) = (v(s,t), r(s,t)) \in \partial K \times [1-r_1,1+4r_1]$ for all $(s,t) \in \RR \times S^1$.\vspace{0.2cm} 

\textit{Case 1.} The orbits $x$ and $y$ are constant.\vspace{0.2cm}

In this case, the topological energy of $u$ is given by 
\begin{align*}
    E_{\textrm{top}}(u) &= \mathcal{A}_{\sigma_{\ell+1} H}(y) - \mathcal{A}_{\sigma_{\ell} H}(x)=(\sigma_{\ell +1} -\sigma_\ell )c_H.  
\end{align*}
Passing to a subsequence if necessary, we may assume that $\sigma_{\ell +1} \ge 2 \sigma_\ell$. Therefore,
\begin{align*}
    E_{\textrm{top}}(u) \ge \sigma_\ell c_H \ge \sigma_1 c_H.
\end{align*}

\textit{Case 2.} The orbit $x$ is constant while the orbit $y$ is nonconstant.\vspace{0.2cm}

We claim that no such Floer trajectory of $H_s$ exists. The argument is essentially the same as in the proof of Proposition \ref{nof}, so we only sketch it here. Write $x=(p,r_x)$ and $y=(\gamma,r_y)$ where $p$ is a point and $\gamma $ is a Reeb orbit of $(\partial K,\alpha)$. Since $x$ is constant whereas $y$ is nonconstant, we have $r_x \ge r_y$. Moreover, the radial coordinate $r(s,t)$ satisfies $ r(s,t) \leq r_x$ since the image of $u$ cannot leave $K_{r_x} =  K \cup (\partial K \times[0,r_x]) $ by Lemma 3.5 of \cite{gt}. Define $f:\RR \to \RR$ by $f(s) = \displaystyle\int_{S^1} \ln r(s,t) dt$. Then, as we have seen in the proof of Proposition \ref{nof}, 
        \begin{align*}
             f'(s) = g(s) + \int_{S^1}  h_s'(r(s,t))dt\,\,\,\,\,\textrm{where}\,\,\,\, g(s) =\int_{S^1} \alpha\left(\partial_t v(s,t)\right) dt.
        \end{align*}
 The function $g(s)$ is nondecreasing from the proof of Proposition \ref{nof}. Since $x$ is a constant orbit, $\displaystyle\lim_{s \to -\infty} \partial_t v = 0$, meaning $\displaystyle\lim_{s \to -\infty} g(s) = 0$. Hence, $g(s) \ge 0$. Since $h_s' \ge 0$, it follows that $f(s)$ is a nondecreasing function and $f(s) \ge \displaystyle\lim_{s\to -\infty} f(s) = \ln r_x$. This implies that $r(s,t) = r_x$ for all $(s,t) \in \RR \times S^1$, which is a contradiction.\vspace{0.2cm}

 \textit{Case 3.} The orbit $x$ is nonconstant while $y$ is constant.\vspace{0.2cm}

 By \eqref{ima}, we have $\mathcal{A}_{\sigma_\ell H}(x) \le c_{\sigma_\ell H} = \sigma_\ell c_H$. Since $\mathcal{A}_{\sigma_{\ell+1}H}(y) = c_{\sigma_{\ell+1}H} = \sigma_{\ell+1} c_H $, we have
 \begin{align*}
    E_{\textrm{top}}(u) &= \mathcal{A}_{\sigma_{\ell+1} H}(y) - \mathcal{A}_{\sigma_{\ell} H}(x)\ge(\sigma_{\ell +1} -\sigma_\ell )c_H. 
 \end{align*}
Assuming $\sigma_{\ell +1} \ge 2 \sigma_\ell$, we obtain
\begin{align*}
    E_{\textrm{top}}(u) \ge \sigma_\ell c_H \ge \sigma_1 c_H.
\end{align*}

\textit{Case 4.} The orbits $x$ and $y$ are nonconstant.\vspace{0.2cm}

Write $x=(\gamma_x,r_x)$ and $y = (\gamma_y, r_y)$ for $r_x,r_y \in (1+2r_1,1+3r_1)$. Proposition \ref{pe} implies that 
\begin{align}\label{p1}
    \sigma_\ell h'(r_x ) =\mathcal{A}(\gamma_x) \geq  \mathcal{A}(\gamma_y)= \sigma_{\ell+1} h'(r_y ).  
\end{align}
Because $\sigma_\ell \leq \sigma_{\ell+1}$, it follows that
\begin{align*}
    h'(r_x) \ge h'(r_y).
\end{align*}
By the concavity of $h$ on $[1+2r_1, 1+3r_1]$, we have $r_x \leq r_y$ and hence 
\begin{align}\label{p2}
    h(r_x) \leq h(r_y).
\end{align} 
Then the action of $y$ satisfies that
\begin{align*}
    \mathcal{A}_{\sigma_{\ell+1}H}(y) &= \sigma_{\ell+1} (-r_yh'(r_y)+h(r_y))\\
    &\ge  -\sigma_\ell r_y h'(r_x) +\sigma_{\ell+1} h(r_y)&&\textrm{by \eqref{p1}}\\
    &\ge -\sigma_\ell r_y h'(r_x) +\sigma_{\ell+1} h(r_x)&&\textrm{by \eqref{p2}},
\end{align*}
and therefore, the topological energy of $u$ admits the estimate
\begin{align*}
    E_{\textrm{top}}(u) &= \mathcal{A}_{\sigma_{\ell+1}H}(y) - \mathcal{A}_{\sigma_{\ell}H}(x)\\
    &\ge -\sigma_\ell r_y h'(r_x) +\sigma_{\ell+1} h(r_x) - \sigma_\ell (-r_x h'(r_x)+h(r_x))\\
    &=-\sigma_\ell h'(r_x)(r_y -r_x) +(\sigma_{\ell+1} - \sigma_\ell) h(r_x)\\
     &\ge -\sigma_\ell s_H r_1 +(\sigma_{\ell+1} - \sigma_\ell) h(r_x)
\end{align*}
where the last inequality follows from the facts that $h'(r_x) \leq s_H$ and $r_y-r_x \leq r_1$. We may choose an increasing sequence $\{\sigma_\ell\}$ such that $\sigma_{\ell+1} \geq 2 \sigma_\ell$. Then 
\begin{align*}
    E_{\textrm{top}}(u) \geq \sigma_\ell(h(r_x)-r_1s_H). 
\end{align*}
Since $h(r_x) \geq h(1+r_1) + r_1 s_H$, we have 
\begin{align*}
    E_{\textrm{top}}(u) \geq \sigma_\ell h(1+r_1). 
\end{align*}

Combining all four cases above, we may take $$c = \min\{\eta_1, \sigma_1 c_H,\sigma_1 h(1+r_1)\}= \min\{\eta_1, \sigma_1 h(1+r_1)\}.$$ This completes the proof.\\
\end{proof}

\begin{lemma}\label{lem2}
    For any $K$-semi-admissible Hamiltonian function $H$ and for any $\tau\in\RR$, we have
\begin{align*}
  \widehat{\varinjlim_{\sigma \to \infty}}CF^{>-\tau}_{L}(\sigma H) \cong \varinjlim_{\sigma \to \infty}CF^{>-\tau}_{L}(\sigma H).
\end{align*}
\end{lemma}
\begin{proof}
    Let $\{\sigma_\ell\}$ be an increasing sequence of positive numbers such that $\displaystyle\lim_{\ell\to \infty} \sigma_\ell = \infty$ and $\sigma_\ell s_H \notin \textrm{Spec}(\partial K, \alpha)$. The generators of $CF^{>-\tau}_L(\sigma_\ell H)$ consist of critical points of $H$ inside $K$ and nonconstant 1-periodic orbits of $\sigma_\ell H$ lying in $\partial K \times [1,1+r_1]$ corresponding to Reeb orbits of $(\partial K, \alpha)$ with periods less than $\tau$. Since there are only finitely many such Reeb orbits, the number of nonconstant generators of $CF^{>-\tau}_L(\sigma_\ell H)$ stabilizes for sufficiently large $\ell$. Furthermore, we may assume that the perturbations of $\sigma_\ell H$ and $\sigma_{\ell+1}H$ have the same critical points on $S^1\times K$; for instance, the perturbation of $\sigma_\ell H$ may be chosen to be $\frac{1}{2^\ell}f$ for some fixed $C^2$-small negative Morse function $f$ on $K$. It follows that the complexes $CF^{>-\tau}_L(\sigma_\ell H)$ and $CF^{>-\tau}_L(\sigma_{\ell+1} H)$ have the same number of generators for sufficiently large $\ell$ and therefore both complexes can be regarded as $\Lambda_{\ge 0}^k$ for some integer $k \ge 1$.
    
    We claim that the continuation map, for sufficiently large $\ell\ge 1$,   
     $$c_L^\ell =c_L^{\sigma_\ell H, \sigma_{\ell+1}H} : CF^{>-\tau}_L(\sigma_\ell H) \lr CF^{>-\tau}_L(\sigma_{\ell+1} H)$$
     has the matrix form 
\begin{align*}
c^\ell_L = D_\ell + T^{\lambda_\ell} B_\ell    
\end{align*}
     where $D_\ell = \textrm{diag}\left(T^{\epsilon_{\ell,1}}, \cdots, T^{\epsilon_{\ell,k}} \right)$ is a diagonal matrix with $\sum_\ell \epsilon_{\ell,i} < \infty$ for each $i$, $\lambda_\ell \geq \max_i \epsilon_{\ell,i}$, and $B_\ell $ is a strictly upper triangular matrix whose entries are elements of $\Lambda_{\ge 0}$. Then the result follows from Lemma 5.5 of \cite{dgpz}.

     Note that the continuation map $c^\ell_L$ sends each constant orbit to itself. Moreover, $c^\ell_L$ cannot map a constant orbit to a nonconstant orbit of $\sigma_{\ell+1}H$, since the action increases along Floer trajectories.
    
    Let $u : \RR \times S^1 \lr M$ be a Floer trajectory connecting a nonconstant orbit $(\gamma_x,r_x)$ of $\sigma_\ell H$ to an orbit $y$ of $\sigma_{\ell+1}H$ for $r_x \in (1,1+r_1)$ and $\mathcal{A}(\gamma_x) < \tau$. If $y$ is a constant orbit, then, as observed earlier,
     \begin{align*}
         E_{\textrm{top}}(u) &= -\mathcal{A}_{\sigma_\ell H}(x) \ge \mathcal{A}(\gamma_x) \ge \min \textrm{Spec}(\partial K, \alpha).
        \end{align*}
     Hence, $E_{\mathrm{top}}(u)$ is bounded below by a fixed positive constant which is independent of $u$.

             Now assume that $y = (\gamma_y,r_y)$ is a nonconstant orbit of $\sigma_{\ell+1}H$ for $r_y \in (1,1+r_1)$ and $\mathcal{A}(\gamma_y)< \tau$. Since 
     \begin{align*}
         \mathcal{A}(\gamma_x) = \sigma_\ell h'(r_x)\,\,\textrm{and}\,\,\mathcal{A}(\gamma_y) = \sigma_{\ell+1} h'(r_{y}),
     \end{align*}
     the numbers $r_x$ and $r_y$ depend on $\ell$ and we write $r_x = r_{x,\ell}$ and $r_y = r_{y,\ell}$. Obviously, 
     \begin{align}\label{lim1}
         \lim_{\ell\to \infty}r_{x, \ell} = 1\,\,\textrm{and}\,\, \lim_{\ell\to \infty}r_{y, \ell} = 1.
     \end{align}
     Moreover, since
     \begin{align*}
        0\leq \sigma_\ell h(r_{x,\ell}) = \sigma_\ell \int_{1}^{r_{x,\ell}} h'(r)dr
         \leq \sigma_\ell (r_{x,\ell}-1)h'(r_{x,\ell}) =  (r_{x,\ell}-1) \mathcal{A}(\gamma_x),
     \end{align*}
     we have
     \begin{align}\label{lim2}
         \lim_{\ell \to \infty} \sigma_\ell h(r_{x,\ell}) = 0 \,\,\textrm{and}\,\, \lim_{\ell \to \infty} \sigma_{\ell+1} h(r_{y,\ell}) = 0.
     \end{align}

      Let $\gamma_1, \gamma_2,\cdots,\gamma_k$ be Reeb orbits of $(\partial K,\alpha)$ whose periods are less than $\tau$. First, we consider the case where $\gamma_x = \gamma_y =\gamma_i$ for some $i$. For notational convenience, we write $r_{x,\ell} = r_\ell$ and $r_{y,\ell} = r_{\ell+1}$. Then 
      \begin{align*}
          E_{\textrm{top}}(u) &= \sigma_{\ell+1}(-r_{\ell+1} h'(r_{\ell+1})+h(r_{\ell+1})) - \sigma_\ell(-r_{\ell} h'(r_{\ell})+h(r_{\ell}))\\
          &=\mathcal{A}(\gamma_i) (r_\ell- r_{\ell+1})+ \{\sigma_{\ell+1}h(r_{\ell+1})-\sigma_\ell h(r_\ell)\}.
      \end{align*}
     Note that 
    \begin{align*}
        \sum_{\ell=1}^\infty \left( \mathcal{A}(\gamma_i) (r_\ell- r_{\ell+1})+ \sigma_{\ell+1}h(r_{\ell+1})-\sigma_\ell h(r_\ell)\right) = r_1\mathcal{A}(\gamma_i) - \sigma_1 h(r_1)< \infty
    \end{align*}
    by \eqref{lim1} and \eqref{lim2}. 
    
    We now consider the case where $\gamma_x =\gamma_i$ and $\gamma_y =\gamma_j$ for $i \ne j$. If $\mathcal{A}(\gamma_x)< \mathcal{A}(\gamma_y)$, then 
    \begin{align}\label{topu}
        E_{\textrm{top}}(u)& = -r_{y,\ell} \mathcal{A}(\gamma_y) + \sigma_{\ell+1}h(r_{y,\ell}) +r_{x,\ell} \mathcal{A}(\gamma_x) - \sigma_{\ell}h(r_{x,\ell})
    \end{align}
    converges to $-\mathcal{A}(\gamma_y) +\mathcal{A}(\gamma_x) < 0$ as $\ell \to \infty$ due to \eqref{lim1} and \eqref{lim2}. This contradicts the nonnegativity of the topological energy. Therefore, we have $ \mathcal{A}(\gamma_x) \geq \mathcal{A}(\gamma_y)$.
   We can further show that $\mathcal{A}(\gamma_x) > \mathcal{A}(\gamma_y)$. If $\mathcal{A}(\gamma_x)= \mathcal{A}(\gamma_y)$, then the topological energy \eqref{topu} converges to zero by \eqref{lim1} and \eqref{lim2}. By Proposition \ref{mon}, we may assume that the image of $u$ is entirely contained in $\partial K \times [1-r_1,1+4r_1]$ for sufficiently large $\ell$. Proposition \ref{pe} then implies that $\gamma_x = \gamma_y$ but this contradicts our assumption that $\gamma_x=\gamma_i \ne \gamma_j = \gamma_y$, and hence
 \begin{align}\label{cope}
        \mathcal{A}(\gamma_x) > \mathcal{A}(\gamma_y).
    \end{align}
    Since $\mathcal{A}(\gamma_x) = \sigma_\ell h'(r_x)> \sigma_{\ell+1} h'(r_{y})=\mathcal{A}(\gamma_y)$ and $\sigma_\ell \le \sigma_{\ell+1}$, it follows that
    \begin{align*}
        h'(r_x) > h'(r_y).
    \end{align*}
    Then by the convexity of $h$ on $[1,1+r_1]$, we obtain
     \begin{align}\label{cor}
         r_{x,\ell}> r_{y,\ell}.
     \end{align}
      Then the topological energy of $u$ satisfies that
    \begin{align*}
        E_{\textrm{top}}(u)& = -r_{y,\ell} \mathcal{A}(\gamma_y) + \sigma_{\ell+1}h(r_{y,\ell}) +r_{x,\ell} \mathcal{A}(\gamma_x) - \sigma_{\ell}h(r_{x,\ell})\\
        & > r_{y,\ell} (-\mathcal{A}(\gamma_y)+\mathcal{A}(\gamma_x))+\sigma_{\ell+1}h(r_{y,\ell}) - \sigma_{\ell}h(r_{x,\ell}) &&\textrm{by \eqref{cor} } \\
        &\ge -\mathcal{A}(\gamma_y)+\mathcal{A}(\gamma_x)+\sigma_{\ell+1}h(r_{y,\ell}) - \sigma_{\ell}h(r_{x,\ell}) &&\textrm{by \eqref{cope} and }r_{y,\ell}\geq 1\\
        &\geq m  +\sigma_{\ell+1}h(r_{y,\ell}) - \sigma_{\ell}h(r_{x,\ell})
    \end{align*}
where 
\begin{align*}
    m = \min \left\{ \mathcal{A}(\gamma_i) -\mathcal{A}(\gamma_j) \bigmid \mathcal{A}(\gamma_i) >\mathcal{A}(\gamma_j) \,\,\textrm{and}\,\,1\leq i,j\leq k  \right\}>0.
\end{align*}
Since $\sigma_{\ell+1}h(r_{y,\ell}) - \sigma_{\ell}h(r_{x,\ell})$ converges to zero as $\ell \to \infty$ by \eqref{lim2}, for sufficiently large $\ell$, we have
\begin{align*}
     E_{\textrm{top}}(u) \ge m.
\end{align*}
Note that the lower bound $m$ does not depend on $u$.

The discussion above shows that the continuation map $c_L^\ell$ has the desired matrix form, and this completes the proof.\\
\end{proof}

\begin{lemma}\label{lem3}
     For any $K$-semi-admissible Hamiltonian function $H$ and for any $\tau\in\RR$, we have
\begin{align*}
  SH^{>-\tau}_M(K) \cong H\left(\displaystyle\varinjlim_{\sigma \to \infty} CF^{>-\tau}_L(\sigma H)\right).
\end{align*}
\end{lemma}
\begin{proof}
      Consider the following short exact sequence of cochain complexes
    \begin{align*}
        0 \lr CF^{>-\tau}_U(\sigma H) \lr CF^{>-\tau}(\sigma H) \lr CF^{>-\tau}(\sigma H)/CF^{>-\tau}_U(\sigma H) = CF^{>-\tau}_L(\sigma H) \lr 0.
    \end{align*}
    From the discussion above, $\{CF^{>-\tau}_U(\sigma H)\}_{\sigma >0}$ and $\{CF^{>-\tau}_L(\sigma H)\}_{\sigma >0}$ form direct systems, where the maps are given by \eqref{conti}. Since the direct limit is an exact functor, the sequence
    \begin{align*}
        0 \lr \varinjlim_{\sigma \to \infty} CF^{>-\tau}_U(\sigma H) \lr \varinjlim_{\sigma \to \infty} CF^{>-\tau}(\sigma H) \lr \varinjlim_{\sigma \to \infty} CF^{>-\tau}_L(\sigma H) \lr 0
    \end{align*}
is exact. The cochain complex $CF^{>-\tau}_L(\sigma H)$ is a free module over $\Lambda_{\geq 0}$ generated by lower orbits of $\sigma H$ and hence is a flat module. As the direct limit preserves the flatness of a module, we have
\begin{align*}
    \text{Tor}_1^{\Lambda_{\geq 0}}\left(\varinjlim_{\sigma \to \infty}CF^{>-\tau}_{L}(\sigma H), A \right) = 0
\end{align*}
for any $\Lambda_{\geq 0}$-module $A$. Taking $A = \Lambda_{\ge0}/\Lambda_{\geq r}$ for $r>0$, the following sequence 
\begin{align*}
     0 \lr \varinjlim_{\sigma \to \infty} CF^{>-\tau}_U(\sigma H)\ton \Lambda_{\geq 0}/\Lambda_{\geq r} &\lr \varinjlim_{\sigma \to \infty} CF^{>-\tau}(\sigma H) \ton \Lambda_{\geq 0}/\Lambda_{\geq r}\\&\lr \varinjlim_{\sigma \to \infty} CF^{>-\tau}_L(\sigma H) \ton \Lambda_{\geq 0}/\Lambda_{\geq r}\lr 0
\end{align*}
remains exact. Applying the Mittag-Leffler theorem for inverse limits, we obtain the following exact sequence
\begin{align}\label{lmc}
     0 \lr \widehat{\varinjlim_{\sigma \to \infty}} CF^{>-\tau}_U(\sigma H) \lr \widehat{\varinjlim_{\sigma \to \infty}} CF^{>-\tau}(\sigma H) \lr \widehat{\varinjlim_{\sigma \to \infty}} CF^{>-\tau}_L(\sigma H) \lr 0.
\end{align}
Lemma \ref{lem1} and Lemma \ref{lem2} imply that the short exact sequence \eqref{lmc} reduces to the isomorphism of complexes
\begin{align}\label{reduce}
      \widehat{\varinjlim_{\sigma \to \infty}}CF^{>-\tau}(\sigma H) \cong \varinjlim_{\sigma \to \infty}CF^{>-\tau}_{L}(\sigma H).
  \end{align}
  Consequently,
  \begin{align*}
      SH^{>-\tau}_M(K) &\cong H\left( \widehat{\varinjlim_{\sigma \to \infty}} CF^{>-\tau}(\sigma H) \right) &&\textrm{by} \,\,\eqref{ivo}\\&\cong H\left( \varinjlim_{\sigma \to \infty}CF^{>-\tau}_{L}(\sigma H) \right) &&\textrm{by} \,\,\eqref{reduce}\\
      &\cong \varinjlim_{\sigma \to \infty}HF^{>-\tau}_{L}(\sigma H)
  \end{align*}
 where the last isomorphism follows from the fact that cohomology commutes with direct limits. This completes the proof. \\
\end{proof}

\noindent
\textit{Proof of Theorem \ref{a2t}.} For any $\sigma \ge1$, define $f_\sigma: [-i_H,0] \to [-i_{\sigma H},0]$ by $ f_\sigma(a) = \left(a_{\sigma H} \circ a_{ H}^{-1}\right) (a)$. By Lemma 3.5 of \cite{gt}, every Floer trajectory connecting two lower orbits is completely contained in $K_{1+r_1}=K \cup \left(\partial K \times [1,1+r_1]\right)$. Therefore, Proposition 3.1 of \cite{cggm2} yields an isomorphism
\begin{align*}
    HF_L^{>a}(H)\cong HF_L^{>f_\sigma(a)}(\sigma H).
\end{align*}
Moreover, this isomorphism is natural in the sense that they commute with inclusion maps and monotone homotopies. Setting $\tau = a_{ H}^{-1} (a)$, we obtain, for $\tau \in [0,s_H]$,
\begin{align}\label{dss}
    HF_L^{>a_H(\tau)}(H)\cong HF_L^{>a_{\sigma H}(\tau)}(\sigma H).
\end{align}
     Then for $\sigma_1 \le \sigma_2$, we have a map 
     \begin{align}\label{ds}
       HF_L^{>a_{\sigma_1 H}(\tau)}(\sigma_1 H) \lr HF_L^{>a_{\sigma_2 H}(\tau)}(\sigma_2 H)  
     \end{align}
     for $\tau \in [0,s_H]$, and therefore $\left\{ HF_L^{>a_{\sigma H}(\tau)}(\sigma H) \right\}_{\sigma \ge 1}$ forms a direct system together with maps \eqref{ds}. It was shown in \cite{cggm} that 
     \begin{align}\label{l}
         \lim_{\sigma \to \infty} a_{\sigma H}(\tau) = -\tau.
     \end{align}
     Consequently, for $\tau \in [0,s_H]$,
     \begin{align*}
         HF_L^{>a_H(\tau)}(H) &\cong \varinjlim_{\sigma \to \infty} HF_L^{>a_{\sigma H}(\tau)}(\sigma H)&&\textrm{by \eqref{dss}}\\
         &\cong SH^{>-\tau}_M(K) 
     \end{align*}
     where the last isomorphism follows from Lemma \ref{lem3} and the limit \eqref{l}. This finishes the proof of Theorem \ref{a2t}.\\
     \qed

\subsection{Some Floer theoretic barcode entropies}
It follows from Theorem \ref{a2t} that the assignment
$$\tau \mapsto SH^{>-\tau}_M(K;\Lambda)$$ defines a persistence module, whose structure maps
$$SH^{>-\tau_1}_M(K;\Lambda) \lr SH^{>-\tau_2}_M(K;\Lambda)$$ for $\tau_1 \leq\tau_2$ are induced by the inclusion maps. For notational convenience, we will omit $\Lambda$ from the notation throughout the remainder of the paper and simply write $SH_M(K)$ in place of $SH_M(K;\Lambda)$. We denote this persistence module by $SH_M(K)$. Let $\mathcal{B}(SH_M(K))$ be the barcode of the persistence module $SH_M(K)$. We define the \textbf{relative symplectic cohomology barcode entropy} $\hbar(SH_M(K))$ of $SH_M(K)$ by
\begin{align*}
    \hbar(SH_M(K)) = \lim_{\epsilon \to 0} \hbar_{\epsilon} (SH_M(K))
\end{align*}
where
\begin{align*}
    \hbar_{\epsilon} (SH_M(K)) = \limsup_{\sigma \to \infty} \frac{1}{\sigma}\log^+b_{\epsilon}(\textrm{tru}(SH_M(K), \sigma)).
\end{align*}

In view of Theorem \ref{a2t}, the relative symplectic cohomology barcode entropy can be expressed in terms of a single $K$-semi-admissible Hamiltonian function. To this end, for a $K$-semi-admissible Hamiltonian function $H$, define the \textbf{lower Hamiltonian barcode entropy} $ \hbar_L\left(H\right)$ by
\begin{align*}
    \hbar_L\left(H\right)=\lim_{\epsilon \to 0}\, (\hbar_L)_\epsilon\left(H\right)
\end{align*}
where
\begin{align*}
    (\hbar_L)_\epsilon\left(H\right)= \limsup_{\sigma \to \infty} \frac{1}{\sigma} \log^+ b_\epsilon\left( \textrm{tru}\left(HF_L(\sigma H), i_{\sigma H}\right) \right)
\end{align*}
for any $\epsilon>0$.
\begin{proposition}\label{pr3}
For any $K$-semi-admissible Hamiltonian function $H$, 
\begin{align*}
    \hbar\left(SH_M(K)\right) = \frac{1}{s_H}\hbar_L\left(H\right).
\end{align*}
    \end{proposition}
\begin{proof}
    Observe that for any $K$-semi-admissible Hamiltonian function $H$,
    \begin{align*}
        \frac{d}{d \tau} a_H(\tau) &= \frac{d}{d\tau} A_{h_1} \left((h_1')^{-1} (\tau) \right) \\&= \frac{d}{ d\tau} \left( -(h_1')^{-1} (\tau)\tau +h_1\left((h_1')^{-1} (\tau)\right) \right)\\
        &= -\left(\frac{d}{d\tau} (h_1')^{-1} (\tau)\right) \tau - (h_1')^{-1} (\tau) + h_1'\left((h_1')^{-1} (\tau)\right)  \left(\frac{d}{d\tau} (h_1')^{-1} (\tau)\right)\\
        &= - (h_1')^{-1} (\tau)
          \end{align*}
    where we used $h_1'((h_1')^{-1})(\tau)=\tau$. In particular, for $0\leq \tau \leq s_H$,
    \begin{align*}
     -1-r_1 \leq    \frac{d}{d \tau} a_H(\tau) \leq -1.
    \end{align*}
   Hence, we obtain
\begin{align}\label{length}
\tau_2 - \tau_1 \leq a_H(\tau_1) - a_H(\tau_2)
\leq (1+r_1)(\tau_2 - \tau_1).
\end{align}
for $0\leq \tau \leq s_H$. Since \eqref{length} holds for any $H$, it follows that 
\begin{align*}
     \tau_2 - \tau_1 \leq a_{\sigma H}(\tau_1) - a_{\sigma H}(\tau_2) \leq (1+r_1)(\tau_2 - \tau_1)
\end{align*}
for $0\leq \tau_1 \leq \tau_2 \leq s_{\sigma H} = \sigma s_H$ for any $\sigma >0$. These inequalities, together with Theorem~\ref{a2t}, imply that for any $\epsilon>0$,
\begin{align*}
    b_{(1+r_1)\epsilon}\left(\textrm{tru}\left(HF_L(\sigma H), \sigma i_H\right)\right) \leq b_{\epsilon}\left(\textrm{tru}\left(SH_M(K), \sigma s_H\right)\right)\leq b_{\epsilon}\left(\textrm{tru}\left(HF_L(\sigma H), \sigma i_H\right)\right).
\end{align*}
Consequently,
\begin{align*}
    \left(\hbar_L\right)_{(1+r_1)\epsilon}(H) \leq s_H \hbar_\epsilon(SH_M(K)) \leq \left(\hbar_L\right)_{\epsilon}(H).
\end{align*}
Taking the limit as $\epsilon \to 0$,
\begin{align*}
    \hbar_L(H) \leq s_H \hbar(SH_M(K))\leq \hbar_L(H),
\end{align*}
which completes the proof.\\
\end{proof}

\section{Proof of the Main result}\label{sec:section4}
In this section, we prove the main result of the paper, saying that the relative symplectic cohomology barcode entropy is bounded from below by the topological entropy of the Reeb flow restricted to a hyperbolic set. The key ingredient in the proof is the following \textit{crossing energy theorem}, which we now state.
\begin{theorem}\label{cros}
    Let $V$ be a locally maximal hyperbolic set for the Reeb flow $\varphi_\alpha^t$ on $(\partial K,\alpha)$. Fix an interval $[a,b] \subset (1, 1+r_1)$. Let $H$ be a $K$-semi-admissible Hamiltonian function such that $h''' \geq 0$ on $[1,b']$ for some $b'>b$ with $b' <1+r_1$. For $\sigma >1$, let $x = (\gamma,r_*)$ and $y$ be lower orbits of $\sigma H$ such that $\gamma$ is contained in $V$ and $r_* \in [a,b]$. Then there exists $\epsilon_0>0$ such that $E_{\textrm{top}}(u) > \epsilon_0$ for any Floer trajectory $u : \RR \times S^1 \lr M$ of $\sigma H$ connecting $x$ and $y$. Moreover, the constant $\epsilon_0 > 0$ is independent of $\sigma, u$ and $x$.
\end{theorem}
\begin{proof}
By Lemma 3.5 of \cite{gt}, the image of $u$ lies in $K \cup \left(\partial K \times [1,1+r_1]\right)$. If the image of $u$ is not entirely contained in $\partial K \times [1-r_1,1+r_1]$, then then by the same monotonicity argument as in Proposition \ref{mon}, we obtain a positive number $\eta_1$, independent of $\sigma$, $u$ and $x$, such that $E_{\textrm{top}}(u) > \eta_1$. On the other hand, if the image of $u$ is completely contained in $\partial K\times[1-r_1,1+r_1]$, then Theorem 5.1 and Theorem 6.1 of \cite{cggm2} imply that there exists a positive constant $\eta_2$, again independent of $\sigma$, $u$ and $x$, such that $E_{\textrm{top}}(u) > \eta_2$. Taking $\epsilon_0 = \min\{\eta_1, \eta_2\}$, the proof is complete. \\ 
\end{proof}

\begin{theorem}\label{mainthm}
    For a compact hyperbolic set $V$ for $\varphi_\alpha^t$, we have 
    \begin{align*}
      h_{\textrm{top}}(\varphi^t_\alpha|_V) \leq \hbar (SH_M(K)). 
    \end{align*}
\end{theorem}
\begin{proof}
By the variational principle for topological entropy (see \cite{kh} and \cite{th}), for every $\delta>0$, there exists a locally maximal hyperbolic invariant set $V'$ for $\varphi_\alpha^t$ such that
\begin{align*}
    h_{\mathrm{top}}(\varphi_\alpha^t|_{V'})
    \geq
    h_{\mathrm{top}}(\varphi_\alpha^t|_{V})-\delta.
\end{align*}
Therefore, after replacing $V$ by $V'$, we may assume that $V$ is not only hyperbolic and invariant, but also locally maximal for the Reeb flow $\varphi_\alpha^t$.

   For $\sigma>0$, let $p_\alpha(\sigma)$ be the number of periodic orbits of $\varphi_\alpha^t|_V$ whose periods are less than or equal to $\sigma$. Since $V$ is locally maximal and hyperbolic, Theorem \ref{alttop} implies that the topological entropy $h_{\textrm{top}}(\varphi_\alpha^t|_V)$ is given by
   \begin{align*}
       h_{\textrm{top}}(\varphi_\alpha^t|_V) = \limsup_{\sigma \to \infty} \frac{1}{\sigma} \log^+ p_\alpha(\sigma).
   \end{align*}
   Let $H$ be a $K$-semi-admissible Hamiltonian function satisfying $h'''\ge 0$ on $[1,1+r_1]$. Fix an interval $(a,b] \subset (1,1+r_1)$. Let $p_H(\sigma)$ be the number of nonconstant 1-periodic Hamiltonian orbits $(\gamma,r_*)$ of $\sigma H$ such that the Reeb orbit $\gamma$ lies in $X$ and $r_* \in (a,b]$. Note that $p_H(\sigma) \le p_\alpha(\sigma s_H)$ because every 1-periodic orbit of $\sigma H$ corresponds to a Reeb orbit with period less than $\sigma s_H$. Moreover, 
   \begin{align}\label{h}
       p_H(\sigma) = p_\alpha(\sigma h'(b)) - p_\alpha(\sigma h'(a))
   \end{align}
   because the radial coordinate satisfies $r_* \in (a,b]$. Observe that
\begin{align*}
  & h'(b) h_{\textrm{top}}(\varphi_\alpha^t|_V)\\ &= \limsup_{\sigma \to \infty} \frac{1}{\sigma} \log^+ p_\alpha( \sigma h'(b)) \\
    &\overset{(*)}{\leq} \limsup_{\sigma \to \infty} \frac{1}{\sigma} \log^+ \max\{ 2( p_\alpha(\sigma h'(b)) - p_\alpha( \sigma h'(a))), 2 p_\alpha( \sigma h'(a))\}\\
    & = \max \left\{ \limsup_{\sigma \to \infty} \frac{1}{\sigma} \log^+ \left(2( p_\alpha(\sigma h'(b)) - p_\alpha( \sigma h'(a))\right),  \limsup_{\sigma \to \infty} \frac{1}{\sigma} \log^+ 2 p_\alpha( \sigma h'(a)) \right\}\\
    &= \max \left\{ \limsup_{\sigma \to \infty} \frac{1}{\sigma} \log^+ 2p_{H} (\sigma), \limsup_{\sigma \to \infty} \frac{1}{\sigma} \log^+ 2 p_\alpha( \sigma h'(a))\right\} &&\textrm{by}\,\,\eqref{h}\\
    &= \max \left\{ \limsup_{\sigma \to \infty} \frac{1}{\sigma} \log^+ p_{H} (\sigma), h'(a) h_{\textrm{top}}(\varphi_\alpha^t|_V) \right\}. 
\end{align*}
The inequality $(*)$ follows from the elementary observation that
\begin{align*}
    x=(x-y)+y\leq \max\{2(x-y),2y\}
\end{align*}
for all $x,y\in\RR$. Since $h$ is strictly convex on $[1,1+r_1]$, we have $h'(a)< h'(b)$, and hence,
\begin{align}\label{topp}
    h'(b) h_{\textrm{top}}(\varphi_\alpha^t|_V) \le\displaystyle \limsup_{\sigma \to \infty} \frac{1}{\sigma} \log^+ p_{H} (\sigma).
\end{align}

Let $\epsilon_0 >0$ be the constant from Theorem \ref{cros}. Then every Floer trajectory of $\sigma H$ which is asymptotic to a 1-periodic orbit $(\gamma,r_*)$ of $\sigma H$ such that $\gamma$ is contained in $V$ and $r_* \in (1,1+r_1)$ has topological energy greater than $\epsilon_0$ by Theorem \ref{cros}. Hence, for every $\epsilon \in (0,\epsilon_0)$, we have 
\begin{align}\label{bp}
    b_\epsilon\left(\textrm{tru}(HF_L(\sigma H),i_{\sigma H})\right) \ge \displaystyle\frac{1}{2} p_H(\sigma)
\end{align}
by Proposition 3.8 of \cite{cgg}.

Now we are ready to complete the proof. Since the assignment $\epsilon \mapsto b_\epsilon\left(\textrm{tru}(HF_L(\sigma H),i_{\sigma H})\right)$ is decreasing in $\epsilon$, it follows that
\begin{align*}
   \hbar (SH_M(K)) &= \frac{1}{s_H} \hbar_L(H) &&\textrm{by Proposition \ref{pr3} }\\& \geq \frac{1}{s_H} \limsup_{\sigma \to \infty} \frac{1}{\sigma}\log^+ b_\epsilon\left(\textrm{tru}(HF_L(\sigma H),i_{\sigma H})\right)
\end{align*}
for every $\epsilon >0$. Choose $\epsilon>0$ sufficiently small so that \eqref{bp} holds. Then
\begin{align*}
    \hbar (SH_M(K))& \geq \frac{1}{s_H} \limsup_{\sigma \to \infty} \frac{1}{\sigma}\log^+ b_\epsilon\left(\textrm{tru}(HF_L(\sigma H),i_{\sigma H})\right)\\
    & \ge \frac{1}{s_H} \limsup_{\sigma \to \infty} \frac{1}{\sigma}\log^+ \frac{1}{2} p_H(\sigma)&&\textrm{by \eqref{bp}}\\
    &\ge \frac{1}{s_H} h'(b) h_{\textrm{top}}(\varphi_\alpha^t|_V) &&\textrm{by \eqref{topp}}.
\end{align*}
Since the interval $(a,b]\subset(1,1+r_1)$ was arbitrary, we may let $b\to1+r_1$, in which case $ h'(b)\to s_H$, and hence the theorem follows.\\
\end{proof}

\Addresses

\begin{thebibliography} {99}
\bibitem{a} J. Ahn, $S^1$-equivariant relative symplectic cohomology and relative symplectic capacities, preprint,  arXiv:2410.01977.

\bibitem{a2} J. Ahn, Barcode entropy and relative symplectic cohomology, preprint, 	arXiv:2601.15606.

\bibitem{a22} J. Ahn, Comparison of symplectic capacities, preprint, arXiv:2504.10431.

%\bibitem{bo} F. Bourgeois and A. Oancea, An exact sequence for contact- and symplectic homology, Inventiones Mathematicae, 175 (2009), no. 3, 611–680.

\bibitem{cgg}  E. Cineli, V. Ginzburg and B. Gurel, Topological entropy of Hamiltonian diffeomorphisms: a persistence homology and Floer theory perspective, Mathematische Zeitschrift, 308 (2024), no. 4, Paper No. 73, 38 pp.

\bibitem{cggm} E. Cineli, V. Ginzburg, B. Gurel and M. Mazzucchelli, On the barcode entropy of Reeb flows, Selecta Mathematica. New Series,
 31 (2025), no. 4, Paper No. 64, 36 pp.

\bibitem{cggm2} E. Cineli, V. Ginzburg, B. Gurel and M. Mazzucchelli, Invariant sets and hyperbolic closed Reeb orbits, preprint, arXiv:2309.04576

%\bibitem{co}  K. Cieliebak and A. Oancea, Symplectic homology and the Eilenberg–Steenrod axioms, Algebraic and Geometric Topology, 18 (2018), no. 4, 1953 – 2130.

\bibitem{csgo} F. Chazal, V. de Silva, M. Glisse and S. Oudot, The structure and stability of persistence modules, SpringerBriefs in Mathematics, Springer, [Cham], 2016. x+120 pp, ISBN:978-3-319-42543-6, ISBN:978-3-319-42545-0

%\bibitem{cz} C. Conley and E. Zehnder, Morse-type index theory for flows and periodic solutions for Hamiltonian equations, Communications on Pure and Applied Mathematics, 37 (1984), no. 2, 207 – 253.

\bibitem{dgpz} A. Dickstein, Y. Ganor, L. Polterovich and F. Zapolsky, Symplectic topology and ideal-valued measures, Selecta Mathematica. New Series, 30 (2024), no. 5, Paper No. 88, 92 pp.

\bibitem{f} R. Fernandes, Barcode entropy and wrapped Floer homology, prerprint, arXiv:2410.05528.

\bibitem{f2} R. Fernandes, Wrapped Floer homology and hyperbolic sets, prerprint, arXiv:2501.06654.

\bibitem{fls} E. Fender, S. Lee and B. Sohn, Barcode entropy for Reeb flows on contact manifolds with Liouville fillings, preprint, arXiv:2305.04770.

%\bibitem{g} M. Gromov, Entropy, homology and semialgebraic geometry, Séminaire Bourbaki, Vol. 1985/86, Astérisque No. 145-146 (1987), 5, 225–240.

%\bibitem{gg14} V. Ginzburg and B. Gurel, Hyperbolic fixed points and periodic orbits of Hamiltonian diffeomorphisms, Duke Mathematical Journal, 163 (2014), no. 3, 565–590.

%\bibitem{gg18}  V. Ginzburg and B. Gurel, Hamiltonian pseudo-rotations of projective spaces, Inventiones Mathematicae, 214 (2018), no. 3, 1081–1130.

\bibitem{ggm} V. Ginzburg, B. Gurel and M. Mazzucchelli, Barcode entropy of geodesic flows, to appear in Journal of the European Mathematical Society. 

\bibitem{gt} Y. Ganor and S. Tanny, Floer theory of disjointly supported Hamiltonians on symplectically aspherical manifolds, Algebraic and Geometric Topology, 23 (2023), no. 2, 645 – 732.

\bibitem{kh} A. Katok and B. Hasselblatt, Introduction to modern theory of dynamical systems, Encyclopedia of Mathematics and its Applications, 54, Cambridge University Press, Cambridge (1995) xviii+802 pp, ISBN:0-521-34187-6.

%\bibitem{l} Y. Long, Precise iteration formulae of the Maslov-type index theory and ellipticity of closed characteristics, Advances in Mathematics, 154 (2000), no. 1, 76–131.

%\bibitem{l2} Y. Long, Index theory for symplectic paths with applications, Progress in Mathematics,  207, Birkhäuser Verlag, Basel, (2002) xxiv+380 pp, ISBN:3-7643-6647-8.

\bibitem{m} M. Meiwes, On the barcode entropy of Lagrangian submanifolds, preprint, arXiv:2401.07034.

\bibitem{ms} D. McDuff and D. Salamon, $J$ -holomorphic curves and quantum cohomology, University Lecture Series, 6, American Mathematical Society, Providence, RI, 1994. viii+207 pp.


\bibitem{msv} C.Y. Mak, Y. Sun and U. Varolgunes, A characterization of heaviness in terms of relative symplectic cohomology, Journal of Topology, 17 (2024), no. 1, Paper No. e12327, 26 pp.

%\bibitem{p} J. Pardon, Contact homology and virtual fundamental cycles, Journal of American Mathematical Society, 32 (2019), no. 3, 825–919.

\bibitem{prsz} L. Polterovich, D. Rosen, K. Samvelyan and J. Zhang, Topological persistence in geometry and analysis, University Lecture Series, 74, American Mathematical Society, Providence, RI, (2020), ©2020. xi+128 pp, ISBN:978-1-4704-5495-1. 

\bibitem{ps} L. Polterovich, E. Shelukhin, Autonomous Hamiltonian flows, Hofer's geometry and persistence modules, Selecta Mathematica. New Series, 22 (2016), no. 1, 227 – 296.

%\bibitem{sala2} D. Salamon, Morse theory, the Conley index and Floer homology, The Bulletin of the London Mathematical Society, 22 (1990), no. 2, 113–140.


\bibitem{sala} D. Salamon, Lectures on Floer homology, Symplectic geometry and topology (Park City, UT, 1997), IAS/Park City Mathematics Series, 7, 143 – 229.

%\bibitem{sz} D. Salamon and E. Zehnder, Morse theory for periodic solutions of Hamiltonian systems and the Maslov index, Communications on Pure and Applied Mathematics, 45 (1992), 1303 – 1360.

\bibitem{sun} Y. Sun, Index bounded relative symplectic cohomology, Algebraic and Geometric Topology, 24 (2024), no. 9, 4799–4836.

\bibitem{th} F. Todd and B. Hasselblatt, Hyperbolic flows, Zurich Lectures in Advanced Mathematics, EMS Publishing House, Berlin, [2019], ©2019. xiv+723 pp, 
ISBN:978-3-03719-200-9.

%\bibitem{uz} M. Usher and J. Zhang, Persistent homology and Floer-Novikov theory, Geometry and Topology, 20 (2016), no. 6, 3333–3430.

\bibitem{v} U. Varolgunes, Mayer–Vietoris property for relative symplectic cohomology, Geometry and Topology, 25 (2021), 547 - 642. 

%\bibitem{vt} U. Varolgunes, Mayer–Vietoris property for relative symplectic cohomology, Ph.D. thesis, Massachusetts Institute of Technology.

%\bibitem{y} Y. Yomdin, Volume growth and entropy, Israel Journal of Mathematics, 57 (1987), 285 - 300.




\end{thebibliography}
\end{document}